\documentclass[11pt,oneside,reqno,letterpaper]{amsart}
\usepackage[margin=1in,headsep=0.6in,footskip=0.8in]{geometry}
\usepackage{graphicx}
\usepackage{comment}
\usepackage{bm}
\usepackage{subcaption}
\usepackage[dvipsnames]{xcolor}
\usepackage[T1]{fontenc}
\usepackage{amsmath,amssymb,amsthm}
\usepackage[english]{babel}
\usepackage{enumitem}
\usepackage{amsfonts}
\usepackage{times}
\usepackage{bbm}

\usepackage{hyperref}
\hypersetup{
 colorlinks,
 citecolor=Maroon,
 linkcolor=Maroon,
 urlcolor=Maroon}

\AtBeginDocument{%
   \setlength\abovedisplayskip{11pt}
   \setlength\belowdisplayskip{12pt}}
\makeatletter
\def\@seccntformat#1{\hspace*{0mm}%
  \protect\textup{\protect\@secnumfont
    \ifnum\pdfstrcmp{subsection}{#1}=0 \bfseries\fi
    \csname the#1\endcsname
    \protect\@secnumpunct
  }%
}
\makeatother

\newenvironment{nouppercase}{%
  \renewcommand{\uppercasenonmath}[1]{}}{}

\normalfont
      plus .8\fontdimen3\font
     minus .8\fontdimen4\font
\title[Landau-de Gennes corrections to the Oseen-Frank theory of nematic liquid crystals]{\Large Landau-de Gennes corrections to the Oseen-Frank theory \\ of nematic liquid crystals}
\author{ {\sc G. Di Fratta}}
\address{{{ \small Giovanni Di Fratta},\,{{ \small Institute for Analysis \\
and Scientific Computing, TU Wien\\
Wiedner Hauptstra{\ss}e 8-10 \\
1040 Wien, Austria.}}}}

\author{{\sc J.M. Robbins}}
\address{{{\small Jonathan M. Robbins},\,{\small  School of Mathematics \\
University of Bristol \\
University Walk, Bristol \\
BS8 1TW, United Kingdom.}}}

\author{{\sc V. Slastikov}}
\address{{{\small Valeriy Slastikov},\,{ \small School of Mathematics \\
University of Bristol \\
University Walk, Bristol \\
BS8 1TW, United Kingdom.}}}

\author{{\sc A. Zarnescu}}
\address{{{\small Arghir Zarnescu}, \,{\small  IKERBASQUE, Basque Foundation for Science, Maria Diaz de Haro 3,
48013, Bilbao, Bizkaia, Spain.}} {  BCAM,  Basque  Center  for  Applied  Mathematics,  Mazarredo  14,  E48009  Bilbao,  Bizkaia,  Spain.}{{ ``Simion Stoilow" Institute of the Romanian Academy, 21 Calea Grivi\c{t}ei, 010702 Bucharest, Romania.} }}

\DeclareMathOperator{\diag}{diag}

\def\Vc#1#2{V_{#2}}
\def\pseudomin{{almost-minimising}}

\def\NN{{\mathbbmss N}}
\def\ZZ{{\mathbbmss Z}}
\def\RR{{\mathbbmss R}}

\def\CC{{\mathbb C}}
\def\Sphere{{\mathbb S}}

\newcommand{\myrho}{\rho}

\newcommand{\sgn}{\operatorname{sgn}}
\renewcommand{\Re}{\operatorname{Re}}
\renewcommand{\Im}{\operatorname{Im}}

\newcommand\half{{\textstyle\frac12}}
\newcommand\third{{\textstyle\frac13}}

\newcommand\grad{{ \nabla}}

\newcommand{\assign}{:=}

\newcommand\SO{\textup{SO}}

\newcommand{\mcN}{{\mycal{N}}}

\def\nor{{\bot}}

\def\e{{\varepsilon}}

\def\mcA{{\mycal A}}
\def\mcB{{\mycal B}}
\def\mcF{{\mycal F}}
\def\mcG{{\mycal G}}

\def\mcH{{\mycal H}}
\def\mcN{{\mycal N}}

\def\W{{\mycal W}}

\def\mcA{{\mycal A}}

\def\eps{{\varepsilon}}

\newtheorem{theorem} {Theorem} [section]
\newtheorem*{theorem*}{Theorem}

\newtheorem{lemma} {Lemma}[section]
\newtheorem{proposition} {Proposition}[section]

\theoremstyle{remark}
\newtheorem{definition}{\sc  Definition\rm}[section]
\newtheorem{remark}{\sc  Remark\rm}[section]

\newcounter{marnote}

\def\D2{\grad^{(2)}}

\DeclareFontFamily{OT1}{rsfs}{}
\DeclareFontShape{OT1}{rsfs}{m}{n}{ <-7> rsfs5 <7-10> rsfs7 <10-> rsfs10}{}
\DeclareMathAlphabet{\mycal}{OT1}{rsfs}{m}{n}

\def\tdeg{{\rm deg\,}}

\def\be{\begin{equation}}
\def\ee{\end{equation}}

\def\tr{{\rm tr}}

\def\mcS{{\mycal{S}}}
\def\mcU{{\mycal{U}}}
\def\mcE{{\mycal E}}

\def\be{\begin{equation}}
\def\ee{\end{equation}}
\def\bea#1\eea{\begin{align}#1\end{align}}

\numberwithin{equation}{section}

\newcommand{\tens}{\otimes}

\usepackage{hyperref}

\begin{document}




\begin{abstract}

We study the asymptotic behavior of the minimisers of the Landau-de Gennes model for nematic liquid crystals in a two-dimensional domain in the regime of small elastic constant.  At leading order in the elasticity constant, the minimum-energy configurations can be described by the simpler Oseen-Frank theory. Using a refined notion of $\Gamma$-development we recover Landau-de Gennes corrections to the Oseen-Frank energy. We provide an explicit characterisation of minimizing $Q$-tensors at this order in terms of optimal Oseen-Frank directors and observe the emerging biaxiality. We apply our results to distinguish between optimal configurations in the class of conformal director fields of fixed topological degree saturating the lower bound for the Oseen-Frank energy. 

\end{abstract}	
\begin{nouppercase}
\maketitle
\end{nouppercase}
\vspace{-0.5cm}
\section{Introduction}

Nematic liquid crystals are the simplest liquid crystalline phase as well as the most widely used in  applications.
Among 
the theoretical models for nematic liquid crystals, the most prevalent in the physics and mathematics literature are the Oseen-Frank \cite{Frank1958liquid} and  Landau-de Gennes theories \cite{DeGennes1995}. The Oseen-Frank theory is the simpler of the two, but fails to describe several characteristic features of nematic liquid crystals, including the isotropic-nematic phase transition, non-orientability of the director field, and the fine structure of  defects. By incorporating additional degrees of freedom, the Landau-de Gennes theory accounts for these features, but is more difficult to solve and analyse.

The main focus of this paper is to establish a fine relation between the two theories, in the weak-elasticity regime and for two-dimensional domains. Employing a refined notion of $\Gamma$-development we obtain an approximate expression for Landau-de Gennes minimisers in terms of Oseen-Frank minimisers accurate to energies through the first two orders in the elasticity constant.  The results are applied to a family of boundary conditions of fixed topological degree which saturate a lower bound on the  leading-order Oseen-Frank energy.  For these boundary conditions, we provide explicit solutions in terms of the Green's function for the Laplacian on the domain, and show that the degeneracy in the Oseen-Frank energy is lifted at the next order.   Below we introduce both theories and discuss the mathematical status of their relationship together with the results of this paper.

\subsection{Landau-de~Gennes and Oseen-Frank theories of liquid crystals}
In the Oseen-Frank theory, the liquid crystalline material is assumed to be in the nematic phase.  
Its configuration in a domain $\Omega\subset\RR^d$, $d=2$ or $d=3$,  is described by a unit-vector field $n:\Omega\to \Sphere^2$, 
called the {\it director field}, which represents the mean orientation of the rod-like constituents of the material and  characterises its optical properties.  In the absence of external fields, the director field is taken to be a minimiser of the Oseen-Frank energy,
\begin{equation}
\label{eq: OF energy }
\mcE_{OF}[n] = \int_\Omega K_1|\nabla\cdot n|^2+K_2 |n\cdot(\nabla\times n)|^2+K_3|n\times (\nabla\times n)|^2, 
\end{equation}
subject to Dirichlet boundary conditions $n|_{\partial \Omega} = n_b$, where the $K_j$'s are material-dependent constants.  For mathematical analysis, the one-constant approximation, $K_1 = K_2 = K_3$, is often adopted, according to which the Oseen-Frank energy reduces to the Dirichlet energy, with harmonic maps as critical points.

One shortcoming of this description is that in certain domains, the director field  $n$ is more appropriately represented by an $\RR P^2$-valued map, 
stemming from the fact that 
orientations $n$ and $-n$ are physically indistinguishable.  In simply-connected domains, a continuous $\RR P^2$-valued map $n$ can be lifted to a continuous $\Sphere^2$-valued map,  in which case we say that $n$ is {\it orientable}.  However, in non-simply-connected domains, this may not hold, in which case we say that $n$ is {\it non-orientable}; see  \cite{BaZa11} for further discussion, where the notion of orientability is extended to  $n \in W^{1,p}(\Omega, \RR P^2)$.

Another difficulty is the description of defect patterns.   These are singularities in the director field, which correspond physically to sharp changes in orientational ordering on a microscopic length scale. It is well known that boundary conditions can force the director field to have singularities.  This occurs, for example, when $\Omega$ is a three-dimensional domain with boundary homeomorphic to $\Sphere^2$ and the boundary map $n_b:\partial \Omega\rightarrow \Sphere^2$ has nonzero degree.  In this case, in spite of the singularity, the infimum Oseen-Frank energy is finite.   The difficulty is more acute when 
the boundary data $n_b:\partial \Omega \rightarrow \RR P^2$ 
is non-orientable.   In this case, the Oseen-Frank energy is necessarily infinite.

The Landau-de Gennes theory resolves these difficulties by introducing additional degrees of freedom. The liquid crystalline material is described by a tensor field $Q:\Omega\to\mcS_{0}$ taking values in the five-dimensional space of 
 $3\times 3$ real symmetric traceless matrices, or $Q$-tensors, denoted
\begin{equation}\label{eq:Qtensorsdef}
\mcS_0=\big\{Q\in \RR^{3\times3}\, :\,  Q=Q^t,\;\; \tr \, Q =0\big\},
\end{equation}
 where $Q^t$ and $\tr \, Q$ denote  the transpose and trace of $Q$ respectively. The $Q$-tensor originates from a microscopic description; it represents the second  (and lowest-order nontrivial) moments of a probability distribution on the space of single-particle orientations, $\Sphere^2$, given that orientations $n$ and $-n$ are equally likely~\cite{DeGennes1995}.

The $Q$-tensor field is taken to be a minimiser of an energy comprised of elastic and bulk terms,
\begin{equation}
\mcF[Q] = 
  \frac{L}{2} \int_{\Omega} |\nabla Q|^2 +\int_{\Omega} f_{bulk}(Q), \qquad Q \in H^1(\Omega, \mcS_0),
	\label{Eq:LdG1}
\end{equation}
where $L$, the elastic constant, is a material parameter.  
For $f_{bulk}$ smooth and sufficiently regular boundary conditions, standard results from the calculus of variations  
imply that $\mcF$ has a smooth minimiser; singularities are absent in the Landau-de Gennes theory.
The bulk potential is required to be invariant under rotations $Q\mapsto RQR^t$, $R \in \SO(3)$,  and is usually taken to be of the form introduced by de Gennes, \footnote{More general bulk potentials $g(\tr \, Q^2, \tr \, Q^3)$ have been studied in the literature; see, e.g., \cite{ball2010nematic,fatkullin2005critical}.  We expect the results presented here to apply more generally to bulk potentials with a unique minimiser (modulo rotations) which is nondegenerate and uniaxial.}
\be
f_{bulk}(Q) = \frac{A}{2} \tr \, Q^2 - \frac{B}{3} \tr \, Q^3 + \frac{C}{4} (\tr \, Q^2)^2 .
\ee
Here $A$, $B$, and $C$, are material parameters, possibly temperature-dependent, with $C > 0$.
From now on we will assume without loss of generality that the coefficients $L,A,B$, and $C$, are non-dimensional;  see, for example, \cite{Gartland2015scalings} and the appendix of \cite{Nguyen13} for suitable non-dimensionalisations. We will focus on the generic case $B\not=0$ but also discuss some aspects of the case $B=0$.

 In the class of spatially homogeneous $Q$-tensors the equilibrium configurations correspond to the minimisers of $f_{bulk}$.  For $A>0$,  the zero $Q$-tensor 
 is a local minimiser, 
 and becomes a global minimiser for $A$  sufficiently large.  
 The zero $Q$-tensor corresponds to the isotropic, or orientationally disordered, phase.  For $A < 0$, the minimisers of $f_{bulk}$ are, generically, a two-dimensional manifold within the larger class of  {\it uniaxial} $Q$-tensors, i.e., $Q$-tensors with a doubly degenerate eigenvalue.
 By identifying $n$, the normalised eigenvector  orthogonal to the degenerate eigenspace, as the director, 
 uniaxial $Q$-tensors correspond to the nematic phase as described within the Oseen-Frank theory.   With $A$ regarded as temperature-dependent, the Landau-de Gennes theory is seen to encompass the observed isotropic-nematic phase transition. 
 
The sign of the degenerate eigenvalue of a uniaxial $Q$-tensor 
coincides with the sign of $B$, and distinguishes two qualitatively different phases.  In terms of the probabilistic interpretation of the $Q$-tensor,
 a positive value of the degenerate eigenvalue corresponds to an ensemble of orientations predominantly orthogonal to the director $n$; this is the {\it oblate uniaxial} phase.    A negative value corresponds to an ensemble of orientations predominantly parallel to  $n$; this is the {\it prolate uniaxial} phase, which describes typical nematic liquid crystals.   
Since our focus is on the nematic phase,  we take $A = -a^2 < 0$, $B = -b^2 < 0$ and $C = c^2 > 0$. The set of minimisers  of 
$f_{bulk}$, which we call the {\it limit manifold}, is given by
\begin{equation} \label{def:limmanifold}
\mcS_*:=\textstyle \left\{Q\in\mcS_0;\ \ Q=s_+\left(n\otimes n -\frac13 I \right), \ \ n \in \Sphere^2\right\},
\end{equation}  
where \be \label{def:s+}
s_+ = \frac{b^2+ \sqrt{ b^4 + 24a^2c^2}}{4c^2}.
\ee The limit manifold is homeomorphic to the real projective plane $\RR P^2$. In the non-generic case $b^2=0$ we have that the limit manifold  is given by 
\be \label{eq:limmanifold4-}
\mcS_0^*:=\{Q\in \mcS_0; |Q|^2= (2/3) s_+^2=a^2/c^2\},
\ee which is homeomorphic to $\Sphere^4$.

The minimum of the bulk energy is given by
\begin{equation}
f_* := f_{bulk}(\mcS_*) = -\frac{a^2}{3} s_+^2 - \frac{2b^2}{27} s_+^3 + \frac{c^2}{9} s_+^4.
\end{equation}
It is convenient to replace $f_{bulk}$ by
\be\label{eq:fb}
\tilde f_{bulk} = f_{bulk} - f_*,
\ee
so that $\tilde f_{bulk}(Q) \geqslant 0$ with $\tilde f_{bulk}(Q) = 0$ if and only if $Q\in \mcS_*$. 



\subsection{State of the art} The Landau-de Gennes theory is usually applied to a system  in which the elastic constant $L$ can be treated as a small parameter.  This is the case when the size of the domain is much larger than a characteristic microscopic length scale
(see, for example, \cite{Gartland2015scalings} and the appendix of \cite{Nguyen13}).   With such systems in mind,  we write $L = \eps^2 \ll 1$ and rescale the energy \eqref{Eq:LdG1} to obtain
\begin{equation}
\mcE_\e [Q] = \int_{\Omega}  \frac{1}{2}|\nabla Q|^2 + \frac{1}{\eps^2} \tilde f_{bulk} (Q)\, , \qquad Q \in H^1(\Omega, \mcS_0)\,,
	\label{Eq:LdG2}
\end{equation}
so that deviation from the limit manifold is penalised. We restrict to differentiable boundary conditions taking values in the limit manifold,
\be\label{bc}
Q_b = Q|_{\partial\Omega} \in C^{1}(\partial\Omega, \mcS_*) ;
\ee 
indeed, boundary conditions violating this restriction induce a boundary layer of width $\eps$.
We say that the boundary conditions are {\it orientable} if
\begin{equation}\label{eq: orientable bcs}
\textstyle Q_b =  s_+\left(n_b\otimes n_b-\frac 13 I\right), \ \text{ where } n_b \in C^{1}(\partial\Omega , \Sphere^2).
\end{equation}

It is in the small-$\eps$ regime that the relationship between the Landau-de Gennes and Oseen-Frank theories emerges.  For orientable boundary conditions, if we formally take $\eps = 0$, the Landau-de Gennes energy \eqref{Eq:LdG2} becomes 
\begin{equation}
\label{eq: Gamma}
 \mcE_0 [Q] = \begin{cases} 
\displaystyle \frac{1}{2} \int_{\Omega}  |\nabla Q|^2 &  \text{if } Q \in H^1(\Omega, \mcS_*),\\
+\infty &\text{otherwise}.
\end{cases}
\end{equation} 
Provided 
the domain is simply-connected, given $Q \in H^1(\Omega, \mcS_*)$, there exists $n\in H^1(\Omega, \Sphere^2)$ such that $Q(x)=s_+(n(x)\otimes n(x)-\frac{1}{3}I)$. 
In this case, the limiting energy $\mcE_0 [Q]$ can be expressed  in terms of the director field as
\be \label{eq: Oseen Frank}
 \mcE_{OF} [n]=s_+^2\int_\Omega |\nabla n(x)|^2,
\ee
which is, up to a multiplicative constant, the one-constant Oseen-Frank energy.

There has been much recent work in the mathematics literature analysing the relationship between the two theories in the limit $\eps \rightarrow 0$.   For three-dimensional domains with orientable boundary conditions, it was  shown in \cite{Majumdar10} 
that global minimisers $Q_\e \in H^1(\Omega,\mcS_0)$ 
of $\mcE_\e$ converge to global minimisers $Q_0\ = s_+(n_0\otimes n_0 - \third I) \in H^1(\Omega,\mcS_*)$ 
of $\mcE_0$.  Moreover, outside a finite set of point singularities 
of the one-constant Oseen-Frank director $n_0$, the convergence  holds in strong norms on compact sets.    These results were extended in \cite{Canevari2017line} to the case of non-orientable boundary conditions; the principal new features are (\emph{i}) the Landau-de Gennes energy is logarithmically divergent  in $\eps$, (\emph{ii}) the singular set contains one-dimensional curves as well as isolated points, and (\emph{iii}) the limit map $Q_0$ is described by an $\RR P^2$-valued harmonic map rather than an $\Sphere^2$-valued harmonic map.  Results for two-dimensional domains with more general boundary conditions and assumptions on the behaviour of the energy are given in \cite{Bauman2012analysis, Canevari15biaxiality, Golovaty2014minimizers}.

Given the leading-order behaviour of the Landau-de Gennes minimisers away from singularities, one can pursue two distinct directions.  The first concerns the behaviour of a minimiser $Q_\eps$ 
near the singular set, where deviations from $Q_0$ are no longer small.  This amounts to analysing the  profiles of point and line defects, 
an active area of research \cite{Bauman2012analysis, Canevari15biaxiality, Canevari2017line, canevari2016radial,canevari2015morse, canevari2016defects,  Fratta16, Golovaty2014minimizers, Henao17, INSZ_new, Ignat2015binstability, Ignat2015stability, Ignat20162dstability, kitavtsev2016liquid}.

The second concerns the structure of deviations  $Q_\eps - Q_0$ away from the singular set.  Formal asymptotics suggest  that $Q_{\eps}\sim Q_{0}+\eps^2 P_{\eps}$, where $P_{\eps}$ is $O(\eps^0)$.  This question was addressed in \cite{Nguyen13} for three-dimensional domains with orientable boundary conditions. 
Subject to 
rather restrictive conditions on $Q_0$ (which in particular exclude defects), 
it was shown that $P_\eps$ approaches a limiting map $P_0 \in C_{\mathrm{loc}}^\infty(\Omega, { \mcS_0}) \cap H^{s}(\Omega, { \mcS_0})$ for any $0 < s < 1/2$.  Moreover, $P_0$ splits naturally into a sum $P_0^\bot + P_0^\top$, where $P_0^\top$ takes values in the two-dimensional tangent space $T_{Q_0}\mcS_*$ of $\mcS_*$ at $Q_0$, and  $P_0^\bot$ 
takes values in
the three-dimensional orthogonal complement of $T_{Q_0}\mcS_*$. The transverse component $P_0^\bot$ is given by an explicit expression involving $Q_0$ and its derivatives, while $P_0^\top$ is shown to satisfy a linear inhomogeneous PDE.
\subsection{Contributions of  present work} 
Our results also pertain to corrections to $Q_\eps$ away from the singular set, and complement those of \cite{Nguyen13}. 
Specifically, we consider simply-connected two-dimensional domains with orientable boundary conditions \eqref{eq: orientable bcs} for which the boundary director $n_b$ is planar, i.e., $n_b \cdot e_3 = 0$.   
 By identifying the boundary $\partial \Omega$ with $\Sphere^1$, we may regard a planar boundary director $n_b$  as a map from $\Sphere^1$ to itself, which therefore may be assigned an integer-valued degree, $m$.  We consider the case of nonzero degree.
We use energy-based methods 
to derive an explicit formula for the transverse component of the first-order correction. While we obtain only bounds for the tangential component, and not the linear PDE that it satisfies,   
 we are able to relax the restrictive assumptions on $Q_0$ in \cite{Nguyen13}.
Also, the $\Gamma$-convergence argument is much simpler than the PDE analysis 
of \cite{Nguyen13}, and has  potential  further application to dynamics in terms of a corrected Oseen-Frank energy for the gradient flow. 
 
 Most importantly, the variational analysis  brings to light a physically significant difference between the energies associated with the transverse and tangential  components of $P_0$.  The transverse component, which affects the bulk potential, contributes to the Landau-de Gennes energy at $O(\eps^2)$, while the tangential component, which affects only the elastic energy, contributes  at higher order. 
This observation suggests that the 
transverse component $P_0^\bot$  assumes the same form for a  wide class of $Q$-tensor models in which  the Oseen-Frank theory provides the leading-order description.  Insofar as  the Landau-de Gennes model  is necessarily approximate, this suggests that the transverse component of the $Q$-tensor, while small, is robust under  perturbations; an additional $O(\eps^\delta)$ contribution to the energy produces an  $O(\eps^\delta)$ correction to $P_0^\perp$.
The tangential component   lacks this robustness; an $O(\eps^\delta)$  perturbation typically produces an $O(\eps^{-2+\delta})$ deviation in   $P_0^\top $ (cf.~Remark~\ref{rmk:discussion}).

The additional information contained in the transverse component $P_0^\nor$ is manifested through the resolution (cf.~Remark~\ref{rmk:1stbiax})
\begin{equation}\label{eq: normal component expansion}
P^{\nor}_0 = c_0 Q_0 +c_1  (p_0 \otimes p_0 - q_0 \otimes q_0)+ c_2 (p_0 \otimes q_0 + q_0 \otimes p_0)  ,
\end{equation} 
where $p_0$, $q_0$  constitute an orthonormal basis for the plane perpendicular to the director $n_0$, and $Q_0 = s_+ (n_0 \otimes n_0 - \third I)$.  The $c_0$-term preserves the eigenvalue degeneracy in $Q_0$, and can be regarded as a correction to $s_+$.  The $c_1$- and $c_2$-terms produce a qualitative change in the $Q$-tensor; they break the eigenvalue degeneracy and thereby introduce {\it biaxiality}.   The difference between the two negative  eigenvalues of $Q_\eps$ can be regarded as a measure of biaxiality,  and is given to leading order by $\eps^2(c_1^2 + c_2^2)^{1/2}$,  while the orientation of the associated eigenvectors in the plane orthogonal to $n_0$ is determined by $c_2/c_1$.   It has previously been established that a critical point of the Landau-de Gennes energy is either everywhere uniaxial or else almost everywhere biaxial \cite{Majumdar10, lamy2015uniaxial}. 
The results presented here make this statement quantitative.  


\
Our principal application  is to a special class of planar boundary conditions.  
A standard argument establishes the lower bound  $2\pi |m|$ for the Dirichlet energy of an $\Sphere^2$-valued harmonic map $n$ with degree-$m$ planar boundary conditions.  The lower bound is achieved for  a special family of boundary conditions, which are parameterised  by  $|m|$ arbitrarily located escape points $(a_1,\ldots,a_{|m|}) \in \Omega^{|m|}$ where the director field is vertical, i.e., $n(a_j) = \pm e_3$.  The director field $n$ is conformal with $n\cdot e_3$ sign-definite.  
Conformal director fields
 may be expressed explicitly in terms of the Green's function for the Laplacian on $\Omega$.   The associated textures  are seen to be similar to the well-known Schlieren patterns observed in liquid crystal films (see Figure~\ref{fig:Schlieren}).
%

The degeneracy in the Oseen-Frank energy among these special boundary conditions is lifted by the first-order correction from the Landau-de Gennes energy.  The expression for the first-order correction simplifies in the conformal case, and is proportional to the integral of $|\nabla Q_0|^4$.  Regarded as a potential on $\Omega^{|m|}$,  the first-order energy favours escape points moving to the boundary.  This is illustrated  in the case of the two-disk, for which closed-form expressions are obtained.

For the special case $b^2 = 0$ (as well as more general bulk potentials depending only on $\tr\, Q^2$),  our results can be extended to non-orientable boundary conditions. 
In this case, the minimising set of $f_{bulk}$ is larger than $\mcS_*$; it contains all $Q$-tensors with specified trace norm, and may be identified with $\Sphere^4$.  For finite $\eps$, the Landau-de Gennes energy is equivalent to a Ginzburg-Landau  functional on $\RR^5$-valued maps, which in the $\eps\rightarrow 0$ limit becomes the Dirichlet energy for $\Sphere^4$-valued maps.    For both orientable and non-orientable planar boundary conditions, there is a unique minimising $\Sphere^4$-valued harmonic map  (in the orientable case, it is distinct from the $\mcS_*$-valued minimisers of \eqref{eq: Gamma}), and 
the first-order correction can be expressed in terms of it.  The $\Gamma$-convergence argument is simpler than in the $b^2 > 0$ case.  

\subsection{Outline}
The remainder of the paper is organised as follows. In  Section~\ref{sec:math} 
 we state and discuss our main results on the Landau-de Gennes corrections to the Oseen-Frank energy for the non-degenerate case $b^2 \neq 0$. The proof of the $\Gamma$-development result   (cf. Theorem~\ref{thm:gexp2}) is given in Section~\ref{sec:thm2.1}. In Section~\ref{sec: b = 0 results}, we state and prove Theorem~\ref{prop:zero0 b=0}, which deals with the degenerate case $b^2=0$ and allows for non-orientable boundary conditions. 
  Finally, in Section~\ref{sec: conformal n},  we apply our results to distinguish between optimal configurations in the class of conformal director fields of fixed topological degree that saturate a lower bound on the Oseen-Frank energy.

\section{Statement of main results}
\label{sec:math}

We are interested in studying the minimisers of the Landau-de Gennes energy $\mcE_\e$ in the physically relevant regime $\eps \ll 1$ for the generic case $b^2 > 0$.
Throughout we assume that the domain $\Omega \subseteq \RR^2$ is bounded and simply connected  with $C^1$-boundary.
We consider orientable  planar boundary conditions  with director $n_b \in  C^1(\partial\Omega, \Sphere^2)$, so that $n_b\cdot e_3 = 0$; results for non-orientable planar boundary conditions in the special case $b^2 = 0$ are presented in Section~\ref{sec: b = 0 results}. 

Identifying the space of unit vectors orthogonal to $e_3$ with $\Sphere^1$, and likewise  identifying the domain boundary $\partial \Omega$ with $\Sphere^1$, we may regard 
$n_b$ 
as a map from $\Sphere^1$ to itself, which may be assigned an integer-valued degree. 
Given $n_b$ of nonvanishing degree,
we denote by $\mcU$ the class of admissible $Q$-tensor fields,
\be \label{def:U admissible}
\mcU := \left\{ Q \in H^1(\Omega, \mcS_0), \, \left. Q\right|_{\partial \Omega} =Q_b \right\} \, , \quad Q_b := s_+ \left(n_b \otimes n_b - \third I\right).
\ee
We consider the minimisation problem (cf. \eqref{Eq:LdG2})
\be\label{main_problem}
\min_{Q\in \mcU}  \mcE_\e[Q] = \min_{Q\in \mcU} \int_{\Omega}  \frac{1}{2}|\nabla Q|^2 + \frac{1}{\eps^2} \tilde f_{bulk} (Q)\, .
\ee
As a first step, we need to understand the behaviour of Problem \eqref{main_problem} in the limit $\e \to 0$.  Using methods of $\Gamma$-convergence we obtain the following result, whose proof is standard and therefore omitted.

\begin{proposition} \label{prop:zero}
As $\e \to 0$, the following statements hold:
\begin{enumerate}[leftmargin=1cm,label=\textup{(}\roman*\textup{)},widest = i]
\item  For any family $\{Q_\e\}_{\e>0} \subset \mcU$ such that $\mcE_\e[Q_\e]\leqslant C$  we have, possibly on a subsequence, $Q_\e \rightharpoonup Q$ weakly in $H^1(\Omega, \mcS_0)$  for some $Q \in H^1(\Omega, \mcS_*)$,  where $\mcS_*$ is the limit manifold defined by \eqref{def:limmanifold}.\smallskip
\item The family $(\mcE_\e)_{\e>0}$ $\Gamma$-converges to $\mcE_0$ in the weak topology of $H^1 (\Omega, \mcS_0)$, where
\begin{equation}
\label{eq: Gamma }
 \mcE_0 [Q] = \begin{cases} 
 \displaystyle \frac{1}{2} \int_{\Omega}  |\nabla Q|^2& \text{ if $Q \in H^1(\Omega, \mcS_*)\cap \mcU $},\\
+\infty &\text{otherwise} .
\end{cases}
\end{equation} 
\smallskip
\item The minimisers $\{Q^*_\e\}_{\e>0}$ of the problem \eqref{main_problem} converge strongly in $H^1(\Omega, \mcS_0)$ to the minimisers of the following harmonic map problem:
\be \label{eq:problemH1}
\min_{Q\in \mcU_*}\mcE_0 [Q] \, ,
\ee 
{with  $ \mcU_* \equiv H^1(\Omega, \mcS_*)\cap \mcU$}.
\end{enumerate}
\end{proposition}

\begin{remark}\label{rem: unique minimiser Q_0}
In \cite{sandier1994uniqueness} (see also \cite{INSZ_GL}), it is shown  that Problem \eqref{eq:problemH1}  has precisely two solutions,
\begin{equation} \label{eq:2minimisers}
Q_0^\pm = \textstyle s_+ \left( n_0^\pm \tens n^\pm_0 - \frac13 I \right),
\end{equation} 
where $n^\pm_0\cdot e_3=0$ on $\partial\Omega$, and $n_{0}^+\cdot e_3 > 0$ (resp~$n_{0}^-\cdot e_3 < 0$) in $\Omega$. 
 The vector field $n_0^\pm$ is a smooth  harmonic map with values in $\Sphere^2$ (see, for instance \cite{Helein2d}) and solves the following minimisation problem:
\be \label{eq: min problem for n_0}
\min \left\{ \int_\Omega |\nabla n|^2 \, : \, {n \in H^1(\Omega, \Sphere^2),\  n= n_b \hbox{ on } \partial \Omega}\right\}.
\ee
From now on,  we set $n_0:=n_0^\pm$ and $Q_0:=Q_0^\pm$,  meaning that all the results we state hold for both $n_0^+$ and $n_0^-$.
\end{remark}

\subsection{A refined formulation of asymptotic $\Gamma$-expansion}\label{sec: Gamma convergence result}
The next  step in understanding the link between the Landau-de Gennes and Oseen-Frank theories is the 
asymptotic expansion of the Landau-de Gennes  energy $\mcE_\e$. Using the approach of $\Gamma$-expansion we can obtain a correction to the Oseen-Frank energy and quantify the difference between the two theories. Specifically, with {$n_0:=n_0^\pm\in C^\infty(\Omega,\Sphere^2)\cap C^1(\bar\Omega,\Sphere^2)$} minimising \eqref{eq: min problem for n_0}, we define the {\it renormalised relative energy}
\be\label{eq:Ge}
\mcG_{\e} [Q]:= \frac{1}{\e^2} \left( \mcE_\e[Q] - \mcE_0[Q_0] \right),  \ \  Q \in H^1(\Omega, \mcS_0),
\ee 
and proceed to investigate the  behaviour of minimisers of $\mcG_{\e} $ in $\mcU$. Before stating our main result about $\mcG_{\e}$, a few comments are in order.

The notion of $\Gamma$-expansion was introduced by Anzellotti and Baldo in \cite{anzellotti1993asymptotic}. Their framework permits to derive selection criteria for minimisers when the leading order $\Gamma$-limit manifests degeneracies in the energy landscape. However, our leading order $\Gamma$-limit $\mcE_0$ is not subject to this phenomenon as it admits just the two minimisers \eqref{eq:2minimisers}. This implies that the second-order $\Gamma$-limit will be infinite at every point but $Q_0^\pm$. No matter which (reasonable) topology is considered, the energy will blow up on families that do not converge to $Q_0^\pm$. In order to gain finer details on the convergence behaviour of the minimising sequences, a slightly different approach must be used. We proceed as follows:
\begin{itemize}
\item First, we observe that fairly extended arguments, that are nevertheless straightforward given the existing literature (see, for instance, \cite{Bethuel_1993,INSZ_new, Majumdar10, Nguyen13}), allow to show that if $\{Q^*_\eps\}_{\eps>0}$ is a family of \emph{minimisers} of $\mcE_\e$,  there exists an $\e_0>0$ such that 
\be\label{bd:uniformLipschitz}
\sup_{0<\eps<\eps_0} \|Q^*_\e\|_{W^{1, \infty} (\Omega, \mcS_0)}<\infty\, ,
\ee
and, possibly for a subsequence, $Q^*_\e \to   Q_0^\pm$ strongly in $H^1(\Omega,\mcS_0)$  with $ Q_0^\pm$ one of the two minimisers of problem \eqref{eq:problemH1}.
\smallskip
\item Next, we use \eqref{bd:uniformLipschitz} to deduce fine properties of the minimisers.  We consider all possible families $\{Q_\e\}_{\e>0}$ that behave in a similar way to minimising families (see Definition~\ref{def:am} for the precise formulation), and we provide a description of the limiting  energy capable of distinguishing  different sequences. Then, the second-order $\Gamma$-limit follows as a particular instance of our analysis.
\end{itemize}
The previous considerations motivate the following terminology.

\begin{definition} \label{def:am}  We say that a family $\{Q_\e\}_{\e>0} \subseteq \mcU$ is \emph{\pseudomin}
\  whenever $\{Q_\e\}_{\e>0}$ satisfies the uniform bound \eqref{bd:uniformLipschitz}, and $\mcG_{\e} [Q_{\e}] \leqslant C$ for some constant $C>0$ independent of $\e$.
\end{definition}

\subsection{Main result: the case $b^2>0$.}
Our main result provides detailed information about the expansion of the energy $\mcE_\e$ and is stated in the next Theorem \ref{thm:gexp2}. Before stating it we need to introduce some basic definitions, notation and terminology  that will be used throughout.

The set of $Q$-tensors, $\mcS_0$, is a five-dimensional linear space, with inner product $Q:P = \tr\, (QP)$.  The norm induced by the inner product is denoted by $|Q| := (\tr\, Q^2)^{\frac 12}$.
 It will be convenient to introduce the following orthonormal basis for $\mcS_0$:
\begin{equation}\label{eq:Fbasis}
\begin{gathered} \textstyle
       F_1 = \frac{1}{\sqrt2} \left( e_1 \otimes e_1 -  e_2 \otimes e_2\right), \;\;
\textstyle F_2 = \frac{1}{\sqrt2} \left( e_1 \otimes e_2 +  e_2 \otimes e_1\right),      \;\;  \textstyle    F_3 = \sqrt{\frac32} \left( e_3 \otimes e_3 - \frac13 I \right),       \\
\textstyle      F_4 =\frac{1}{\sqrt2} \left( e_1 \otimes e_3 +  e_3 \otimes e_1\right) , \quad F_5 = \frac{1}{\sqrt2} \left( e_3 \otimes e_2 +  e_2 \otimes e_3\right),
\end{gathered}
\end{equation} with $e_1,e_2,e_3$ standard basis of $\RR^3$.
A tensor $Q\in \mcS_0$ is called {\it biaxial} if all its eigenvalues are distinct. We say that 
$Q$ is {\it uniaxial} if it has a doubly degenerate eigenvalue $-\lambda/3$. In this case, it can be represented uniquely as
\be \label{eq: form of Q_0 for uniaxial}
\textstyle Q= \lambda \left(n\otimes n - \frac13 I\right),
\ee
where $n \in \Sphere^2$ is called the {\it director} and  $\frac23 \lambda \in \RR$ is the (unique) non-degenerate eigenvalue of $Q$. 
More specifically, $Q$ is {\it prolate uniaxial} if $\lambda>0$ and {\it oblate uniaxial} if $\lambda<0$.
Finally,
$Q\in \mcS_0$  is {\it isotropic} if it has a triply degenerate eigenvalue, in which case $Q = 0$. 
If the largest (necessarily positive) eigenvalue of $Q\in \mcS_0$ is nondegenerate, it is called the {\it principal eigenvalue}, and the associated normalised eigenvector is called the {\it principal eigenvector}.  The remaining two eigenvalues of $Q$ (which may be degenerate) and the associated orthonormal eigenvectors are called the {\it subprincipal eigenvalues} and {\it subprincipal eigenvectors}.

We introduce a parameterised family of rotations in $\SO(3)$.  For any $n \in \Sphere^2\setminus \{-e_3\}$, we define
\be\label{eq: R_n}
R_n = I + [e_3 \times n]_\times + \frac{[e_3 \times n]_\times^2}{1+ n \cdot e_3},
\ee
where, for every $\omega\in\RR^3$, the symbol $[\omega]_\times$ denotes the antisymmetric matrix that maps $v\in \RR^3$ to $\omega\times v$.  
It is easy to check that $R_n \in \SO(3)$ and that $R_n\, e_3 = n$.  Indeed, $R_n$ may be uniquely characterised as the rotation about an axis 
 orthogonal to $e_3$
 by an angle  $0 \leqslant\theta < \pi$ that maps $e_3$ into $n$. Note that, when $n\neq e_3$ the axis of rotation is $e_3 \times n$, and the angle of rotation is $\cos^{-1}(n\cdot e_3)$. 
 
\begin{remark} 
 Given a bounded domain $\Omega \subset \RR^2$, we  note the following:  for any $1\leqslant p\leqslant \infty$, if $n \in W^{1,p}(\Omega, \Sphere^2)$ and $1 + n\cdot e_3$ is bounded away from $0$, then $R_n \in W^{1,p}(\Omega, \SO(3))$.
\end{remark} 

In what follows, to shorten notation, we set $Q[n]:=s_+\left( n \otimes  n -\frac13 I \right)$ for any $n\in\Sphere^2$. Also, we set 
\be
\Vc{}{\myrho}:=\sum_{j=1}^3 \myrho_j F_j,\qquad \grad n_0 \otimes \grad n_0:=\sum_{i=1}^2  \partial_i n_0 \otimes  \partial_i n_0.
\ee
Here, $\myrho\in\RR^3$, and the $F_j$ are the first three elements of the basis \eqref{eq:Fbasis}. Note that any $\Vc{}{\myrho}$ has $e_3$ as eigenvector.

\begin{theorem}\label{thm:gexp2} 
As $\e \to 0$, the following assertions hold:
\begin{enumerate}[leftmargin=1.1cm,label=\textup{(}\roman*\textup{)},widest = i]
\item\label{thm2.1i} 
For any family $\{Q_\e\}_{\e>0} \subseteq \mcU$ such that  $\mcG_{\e} [Q_{\e}] \leqslant C$ we have, possibly on a subsequence, $Q_{\e} \to   Q_0^\pm:=s_+\left( n_0^\pm \otimes  n_0^\pm -\frac13 I \right)$ strongly in $H^1(\Omega,\mcS_0)$,
where $ Q_0^\pm$ is one of the two minimisers of problem \eqref{eq:problemH1}.
\smallskip
\item\label{thm2.1ii} If $\{Q_\e\}_{\e>0}$  is \pseudomin\  then, possibly on a subsequence, $Q_\eps \to Q_0=Q[n_0]$ in $H^1(\Omega,\mcS_0)$ and there exists a family of principal eigenvectors of $Q_\eps$, denoted as $n_\e\in W^{1,\infty}(\Omega, \Sphere^2)$, and a vector-valued function $\myrho \in L^2(\Omega, \RR^3)$,  such that  
\begin{equation} \label{eq: n_eps - n_0 estimate}
\| n_\e -n_0 \|_{H^1_0(\Omega,\RR^3)} \leqslant C \e
\end{equation}
and
\begin{equation}
P_\e^\nor:=\frac{1}{\e^2} \left( Q_\eps  - Q[n_\e] \right) \rightharpoonup 
R_{n_0} \Vc{\left(\sum_{j=1}^3 \myrho_j  F_j\right)}{\myrho} R_{n_0}^t \text{\quad weakly  in  $L^2(\Omega, \mcS_0)$,}   \label{eq: normal component of P weak}
\end{equation}
where $R_{n_0}\in W^{1,\infty}(\Omega, \SO(3))$  is the field of rotation matrices given by \eqref{eq: R_n}.
\smallskip

\item\label{thm2.1iii} 
 Let $\{Q_\e\}_{\e>0}$ be an \pseudomin \ family such that $Q_\eps \to Q_0$ in $H^1(\Omega, \mcS_0)$. For any family of principal eigenvectors $n_\e\in W^{1,\infty}(\Omega, \Sphere^2)$ satisfying \eqref{eq: n_eps - n_0 estimate}, and any $P_\e^\nor\in H_0^1(\Omega, \mcS_0)$ satisfying \eqref{eq: normal component of P weak} we have
\begin{equation} \label{eq:mcH_0 defn}
\liminf_{\e\to 0} \mcG_{\e} [Q_{\e}] \geqslant  \mcH_0 [n_0, \myrho]:= \int_\Omega \frac12 B_0 \myrho \cdot \myrho + b_0 \cdot \myrho.
\end{equation}
Here, $B_0 = \hbox{\rm{diag}}(\mu, \mu, \nu)$ with $\mu =  b^2 s_+$, $\nu  =  \frac13 b^2 s_+ + 2a^2$, and $b_0 \in L^\infty(\Omega, \RR^3)$ is defined by
\be 
\begin{aligned} \label{eq:mcH_0 defnb}
  b_0\cdot e_j & :=  -2 s_+ \left(\nabla n_0 \otimes \nabla n_0\right) : ({R_{n_0}} F_j R_{n_0}^t), \ j=1,2, \\
 b_0\cdot e_3&:= \sqrt{6} s_+ |\nabla n_0|^2.
\end{aligned}
\ee
Also, for every $\myrho\in L^2(\Omega, \RR^3)$ there exists a recovery \pseudomin \ family $Q_\eps= Q[n_\e]+\e^2 P_\e^\nor$ satisfying \eqref{eq: n_eps - n_0 estimate}, \eqref{eq: normal component of P weak} for which $\lim_{\e\to 0} \mcG_{\e} [Q_{\e}] =  \mcH_0 [n_0, \myrho]$.
\smallskip
\item\label{thm2.1iv} 
The unique minimiser of $\mcH_0[n_0,\cdot]$ is given by $\myrho_0 := { -B_0^{-1}} \, b_0$. The corresponding minimum value of the energy is given by
 \begin{equation}
 \label{eq:fundlowerboundPo} \mcH_0[n_0,\myrho_0] =  - s_+^2 \int_\Omega  \frac{2}{\mu}  \left| \grad n_0 \otimes \grad n_0 \right|^2 + \left( \frac{3}{\nu} - \frac{1}{\mu}\right) |\nabla n_0|^4.
\end{equation}
In particular, in the topology induced by \eqref{eq: n_eps - n_0 estimate} and \eqref{eq: normal component of P weak}, the family of energies $\{\mcG_\e\}_{\e>0}$ $\Gamma$-converges to $\mcG_0$, where
\be \label{eq:Gammalimitbdiff0}
 \mcG_0[Q]=  \begin{cases} 
 \mcH_0[n_0,\myrho_0]& \text{ if $Q =Q_0$},\\
+\infty &\text{otherwise}.\end{cases} 
\ee
Moreover, {if $Q^*_\e \to Q_0$ in $H^1(\Omega,\mcS_0)$} is a family of minimisers of $\{\mcE_\e\}_{\e>0}$ with principal eigenvectors $n^*_\eps \in W^{1,\infty}(\Omega, \Sphere^2)$, then $\mcG_\eps[Q^*_\eps] \to \mcH_0[n_0,\myrho_0]$, and
we have
 \begin{align}
  \label{eq: n_e - n_0 = o(eps) thm 2.1}
 \hspace{1.5cm} \frac{1}{\e} \left( n^*_\eps - n_0 \right) &\to 0  \quad \text{strongly in } H_0^1(\Omega, \Sphere^2),\hspace{1.5cm}  \\
 \label{eq: normal component minimiser}
\hspace{1.5cm} \frac{1}{\e^2} \left( Q^*_\eps  - Q[n^*_\e]  \right) &\to P_0^\nor := R_{n_0}\Vc{\left(\sum_{j=1}^3 \myrho_j  F_j\right)}{\myrho_0} R_{n_0}^t \quad \text{strongly  in  $L^2(\Omega, \mcS_0)$.}
\end{align}
\end{enumerate}
\end{theorem}
The proof of Theorem~\ref{thm:gexp2} is given in Section~\ref{sec:thm2.1}.  Key components include a quadratic lower bound on the variation of the $\Sphere^2$-valued Dirichlet energy at $n_0$ (Lemma~\ref{lem: minimality of n_0}) and a $Q$-tensor decomposition (Lemma~\ref{prop: parameterisation}) into  a sum of two  terms with a common eigenbasis, one taking values on the limit manifold, and the other taking values transverse to it.  The fact that almost-minimisers $Q_\eps$ have  uniformly bounded finite $W^{1,\infty}$-norm is used to bound $\mcG_\eps$ from below (in fact, finite $W^{1,4}$-norm would suffice).

\begin{remark}\label{rem: boundary layer}
We note that $P_0^\perp$ in \eqref{eq: normal component minimiser} does not vanish on the boundary, so that higher-order corrections to the minimiser $Q_\eps^*$  contain a boundary layer.
\end{remark}

\begin{remark}\label{rmk:1stbiax} The expression for $P_0^\bot$ 
can be written as (cf.~\eqref{eq: normal component expansion})
 \begin{equation}
\label{eq:  P_0^N expand}
P_0^\bot = c_0 Q_0 + c_1 \left(p_0 \otimes p_0 - q_0\otimes q_0\right) + c_2 \left(p_0 \otimes q_0 + p_0\otimes q_0\right),
\end{equation} 
where $p_0 = R_{n_0} e_1$ and $q_0 = R_{n_0} e_2$ (so that $n_0$, $p_0$ and $q_0$ constitute an orthonormal frame),  and
\begin{align}
c_0 &= -\frac{2\sqrt{6}}{\nu} \left|\nabla n_0\right|^2, \label{eq: cj's1 }\\ 
c_1 &= \frac{\sqrt{2} s_+}{\mu}\left(|\nabla n_0\cdot p_0|^2 - |\nabla n_0\cdot q_0|^2\right),\label{eq: cj's2 } \\ 
c_2 &= \frac{2\sqrt{2} s_+}{\mu}(\nabla n_0\cdot p_0)\cdot (\nabla n_0\cdot q_0).\label{eq: cj's3 }
\end{align}
 The coefficients $c_1$ and $c_2$ describe biaxiality;  
 the quantity $\eps^2(c_1^2 + c_2^2)$ is the square of the difference of the two subprincipal eigenvalues of the minimiser $Q^*_\eps$, to leading order in $\eps$.
 The coefficient $c_0$ describes an $O(\eps^2)$ correction to the principal eigenvalue of $Q^*_\eps$. 
 \end{remark}


\begin{remark}
The energy  $\mcH_0[n_0,\myrho]$ distinguishes between various \pseudomin \ families $\{ Q_\e \}$ and gives a non-trivial energy landscape. The $\Gamma$-limit $\mcH_0$ provides a starting point for an asymptotic analysis of $Q$-tensor dynamics under  gradient flow.  The fact that $\mcH_0$ depends only on $\rho$ indicates that the director dynamics is much slower than that of displacements transverse to the limit manifold.  Specifically, for an initial condition
 with $o(\eps)$-displacements from the optimal director and 
 $O(\eps^2)$-displacements from the limit manifold, the time scale for director dynamics is nevertheless longer.
 
 \begin{remark} It is easy to generalise Theorem~\ref{thm:gexp2} to boundary conditions where $n_b \cdot e_3 \geqslant 0$ (or  $n_b \cdot e_3 \leqslant 0)$.  
Moreover, if $n_b\cdot e_3$ is strictly positive (or strictly negative) at some point $x_0 \in \partial \Omega$ then it is not necessary to assume that $n_b$ has nonzero degree (indeed, the degree might not be well defined in this case). 
\end{remark}

\begin{remark} \label{rmk:discussion}
An informal argument suggests that Theorem~\ref{thm:gexp2} may extend to more general $Q$-tensor energy densities 
of the form $|\nabla Q|^2  + \eps^{-2} f(Q) +  \eps^\delta g(Q,\nabla Q)$, where $f$ is any bulk potential minimised by prolate uniaxial $Q$-tensors of fixed norm, and $g$ represents an additional contribution to the energy.
In order that the generalised model reduce to the (one-constant) Oseen-Frank description away from defects, we require that $\delta > 0$.  
Under suitable conditions on $g$,  we  expect the transverse component $P_0^\bot$  to be unaffected by this additional contribution, and Theorem~\ref{thm:gexp2} still to hold but with a rate of convergence of $||n_\eps - n_0||_{H^1}\to 0$ possibly depending on $\delta$. 
%
  The key point is that $P_0^\bot$  should still be given by \eqref{eq: normal component minimiser}, with $\mu$ and $\nu$ the nonvanishing eigenvalues of  the Hessian of $f$ at its  minimum.

The argument may be illustrated by a finite-dimensional proxy for the Landau-de Gennes energy, in which the tensor field $Q$ is replaced by just two quantities: $x$,  a proxy for the director displacement $n-n_0$, where $n$ is the principal eigenvector of $Q$; and $y$, a  
proxy for the transverse component, $Q - Q[n]$.  The proxy energy is given by
\be
\mcE_\eps(x,y) =  \left(\frac{1}{2} p x^2  + q xy + \frac{1}{2} r y^2+ by\right) + \frac{\mu}{2\eps^2} y^2 +  \eps^\delta g(x,y),
\ee 
where $p, \delta, \mu  > 0$.  
The term $(\half p x^2  + q xy + \half r y^2+ by)$  corresponds to the elastic energy expanded about its minimum -- hence   the absence of a term linear in $x$
and the requirement that $p> 0$. 
The term $\half \mu y^2/\eps^2$ corresponds to the bulk potential expanded about its minimum; the absence of terms in $x^2$ and $xy$  reflects  the rotational invariance of the bulk potential. 
 To leading order in $\eps$, the minimiser $(x^*_\eps, y^*_\eps)$ is given by
\be  
x^*_\eps = -\frac{\eps^\delta}{r} \frac{\partial}{\partial y} g(0,0) +\frac{\eps^2}{\mu r} b, \quad y^*_\eps = \eps^2  \frac{b}{\mu}. 
\ee
Thus, the ``transverse component''  $y^*_\eps$ is independent of $g$, while the ``director displacement''  $x^*_\eps$  is driven by $g$, at least for $\delta < 2$. 

\end{remark}

%
  %
\end{remark}

\section{$\Gamma$-expansion: proof of Theorem~\ref{thm:gexp2}}\label{sec:thm2.1}
 
\subsection{Proof of $(i)$: equi-coercivity of the energy functionals (compactness) }  
Here we prove statement $(i)$ of Theorem~\ref{thm:gexp2}.
Consider  a family $\{Q_\e\}_{\e>0} \subset \mcU$ such that $\mcG_{\e} [Q_\eps] \leqslant C$ for some constant $C>0$. It is clear that
\be\label{unifenerg}
 \mcE_\e[Q_\eps]=\int_\Omega \frac12  |\nabla Q_\eps|^2 + \frac{1}{\e^2} \tilde f_{bulk} (Q_\e) \leqslant C\eps^2+\mcE_0 [Q_0].
\ee
In particular, $\{\mcE_\e[Q_\eps]\}_{\e>0}$ is bounded for $\e$ sufficiently small.
Since $\tilde f_{bulk} \geqslant 0$,  there exist a (not relabeled) subfamily $\{Q_{\eps}\}_{\e>0}\subset \mcU$, and a tensor field $Q_*\in H^1(\Omega, \mcS_0)$, such that
\be\label{conv:weakQepsk0}
Q_{\eps}\rightharpoonup Q_*\; \hbox{ in } H^1(\Omega, \mcS_0), \quad \tilde f_{bulk} (Q_{\e}) \to 0\; \hbox{ a.e. in }  \Omega.
\ee
From the above, $ Q_* \in   H^1(\Omega, \mcS_*)$ and $Q_{\eps}\to Q_*$ strongly in $L^2 (\Omega, \mcS_0)$.
By the lower semicontinuity of the norm and the bound \eqref{unifenerg} we obtain
\be\label{est:energcomparison0}
\int_\Omega |\nabla Q_*|^2\, \leqslant \liminf_{\eps\to 0} \int_\Omega |\nabla Q_\eps|^2 \leqslant\lim_{\eps\to 0} (C\eps^2+\mcE_0 [Q_0]) = \mcE_0 [Q_0]=\int_\Omega |\nabla Q_0|^2\, ,
\ee
with $Q_0\in\mathrm{argmin}_{Q\in \mcU}\mcE_0 [Q]$. Therefore, 
\be
Q_*\in\mathrm{argmin}_{Q\in \mcU}\mcE_0 [Q] \quad \text{and}\quad \|\nabla Q_\eps\|_{L^2}\to \|\nabla Q_*\|_{L^2}=\|\nabla Q_0\|_{L^2}.
\ee
 Combining this information with \eqref{conv:weakQepsk0} we conclude that $Q_{\eps}\to Q_*$ strongly in $H^1(\Omega, \mcS_0)$.
Eventually, by Remark~\ref{rem: unique minimiser Q_0},
$Q_*=s_+\left(n_0^\pm \otimes n_0^\pm -\frac13 I \right)$
where $n_0^\pm$ is one of the two minimisers of problem \eqref{eq: min problem for n_0}.



\subsection{Proof of (\emph{ii}): parameterisation of {\pseudomin} families and convergence estimates} 
Here we prove statement (\emph{ii}) of Theorem~\ref{thm:gexp2}. In agreement with Remark \ref{rem: unique minimiser Q_0}, and to fix the ideas, we set $n_0:=n_0^+$ and $Q_0:=Q_0^+=s_+ \left( n_0\otimes n_0 - \frac13 I \right)$. Also, to shorten notation, we set $Q[n]:=s_+\left( n \otimes  n -\frac13 I \right)$ for any $n\in\Sphere^2$, and $\Vc{}{\myrho}:=\sum_{j=1}^3 \myrho_j F_j$ for any vector $\myrho\in\RR^3$, where $F_j$ are the first three elements of the basis \eqref{eq:Fbasis}. 

We show that any almost-minimising family $\{Q_\eps\}_{\eps>0}\subseteq\mcU$ admits a parameterisation in terms of two families of vector fields:
\begin{itemize}[leftmargin=0.8cm]
\item the family $\{n_\e\}_{\e>0}\subseteq W^{1,\infty}(\Omega, \Sphere^2)$ of principal normalised eigenvectors of $\{Q_\eps\}_{\eps>0}$;
\smallskip
\item the family of vector fields $\{\myrho_\eps\}_{\e>0}\subseteq W_0^{1,\infty}(\Omega, \RR^3)$ that characterises the displacement between $Q_\eps$ and  the limit manifold $\mcS_*$ defined by \eqref{def:limmanifold}. 
\end{itemize}
The parameterisation facilitates the fine control of the energy difference $\mcE_\eps[Q_\eps] - \mcE_0[Q_0]$; contributions to $\mcE_\eps[Q_\eps]$ from $\myrho_\eps$  are controlled by the bulk potential, which takes its minimum on the limit manifold, while contributions from $n_\eps$ are controlled by the elastic energy, using Lemma~\ref{lem: minimality of n_0} below.  This separation is necessitated by the fact that, by rotational invariance, the second variation   $\grad^{(2)}   f_{bulk}$ of the bulk potential  on $\mcS_*$  is only positive semidefinite, not positive definite.  To linear order, variations in $n_\eps$ are tangent to  $\mcS_*$ and lie in the null space of $\D2   f_{bulk}$, while variations in $\myrho_\eps$ are normal to $\mcS_*$ and lie in the subspace on which $\D2    f_{bulk}$ is positive definite.

\begin{lemma}\label{prop: parameterisation}
 Let $\Omega\subset\RR^2$ be a bounded and simply-connected domain and  $n\in C^1(\bar\Omega,\Sphere^2)$.  
 Suppose that $\{Q_\eps\}_{\eps>0}\subseteq\mcU$ is uniformly bounded in $W^{1,\infty}(\Omega, \mcS_0)$, and 
 \[
 Q_\e \rightarrow Q[n]:=s_+ \textstyle\left( n\otimes n - \frac13 I \right)\text{\quad strongly in $H^1(\Omega,\mcS_0)$.}
 \]    
Then, for $\e$ sufficiently small, the following hold:
\begin{enumerate}[leftmargin=1cm,label=\textup{(}\roman*\textup{)},widest = i]
  \item There exists a principal eigenvector $n_\eps \in W^{1,\infty}(\Omega, \Sphere^2)$ of $Q_\eps$ such that for any $1 \leqslant p <\infty$,
\begin{equation}
\label{eq: n_e goes to n_0 }
n_\eps \rightarrow n \text{ \;in } W^{1,p} (\Omega,\Sphere^2),\text{ as well as in $C(\bar \Omega, \Sphere^2)$.}
\end{equation}

\medskip
  \item There exists a vector-valued function $\myrho_\e \in W_0^{1,\infty}(\Omega, \RR^3)$ such that
\begin{equation}
\label{eq: parameterisation }
Q_\eps =  Q[n_\e] + \eps^2 P^\nor_\eps,\quad
P^\nor_\eps :=R_{n_\eps} \Vc{\sum_{j=1}^3 \myrho_{\eps j} F_j}{\myrho_\e} R_{n_\eps}^t.
\end{equation}
Here, $R_{n_\eps}$ is the rotation given by \eqref{eq: R_n}.
Moreover, we have, for any $1 \leqslant p <\infty$,
\begin{equation} \label{eq: eta_e goes to 0 }
R_{n_\eps} \rightarrow R_{n_0} \text{ \; in } W^{1,p} (\Omega,\SO(3)), \quad
\eps^2 \myrho_\eps \rightarrow 0 \text{ \;in } W_0^{1,p} (\Omega,\RR^3).  
\end{equation}
as well as, respectively, in $C(\bar \Omega,\SO(3))$ and in $C(\bar \Omega,\RR^3)$.
\end{enumerate}
\end{lemma}

\begin{proof}
 Since $Q_\eps \rightarrow Q[n]$ in $H^1(\Omega, \mcS_0)$ with $Q_\eps$ uniformly bounded in $W^{1,\infty}(\Omega,\mcS_0)$, by interpolation it is clear that $Q_\eps \rightarrow Q[n]$ in $ W^{1,p} (\Omega,\mcS_0)$ for every $1\leqslant p < \infty$, as well as in $ C(\bar \Omega, \mcS_0)$.

\medskip
\noindent (\emph{i}) The tensor field $Q[n]$ has everywhere a principal eigenvalue equal to $2s_+/3$. It follows that for $\eps$ sufficiently small, $Q_\eps$ has everywhere a principal eigenvalue $\lambda_\eps$ with principal eigenvector $n_\eps$ uniquely determined up to a sign. The fact that  $\lambda_\eps$ is nondegenerate implies that the projector $n_\eps\otimes n_\eps$ can be expressed as a smooth function of $Q_\eps$ (see for instance \cite{nomizu}).
Thereby, $Q[n_\eps] \in W^{1,\infty}(\Omega,\mcS_0)$ and  $Q[n_\eps]  \rightarrow Q[n] $ in $H^1(\Omega,\mcS_0)$, as well as uniformly.  In particular, $n_\eps \cdot n \neq 0$ for $\eps$ sufficiently small.  We may then choose the sign of $n_\eps$ so that $n_\eps \cdot n$ is everywhere positive.  For this choice, $n_\eps \in W^{1,\infty}(\Omega, \Sphere^2)$ and $n_\eps \rightarrow n$ in $W^{1,p} (\Omega,\Sphere^2)$, $1 \leqslant p <\infty$, as well as in $C(\bar \Omega, \Sphere^2)$.

\medskip
\noindent (\emph{ii}) Consider the quantity ${ U_\e}:=\e^{-2} R^t_{n_\eps}\left(Q_\eps - Q[n_\e]\right) R_{n_\eps}$. As $n_\eps$ is an eigenvector of  $Q_\eps$ and of $n_\eps\otimes n_\eps$, it follows that $e_3$ is an eigenvector of ${ U_\e}$.  The unit vector $e_3$ is an eigenvector of $Q \in \mcS_0$ if, and only if, $Q$ is a linear combination of $F_1$, $F_2$ and $F_3$. Therefore  ${ U_\e}=V_{\myrho_\e}$ with $\myrho_\e\cdot e_j:=V_\e : F_j$.
Setting 
\be
P_\eps^\nor = R_{n_\eps} V_{\myrho_\e} R^t_{n_\eps},
\ee
 we establish \eqref{eq: parameterisation }.
Next, we prove that $\myrho_\e\in W_0^{1,p}(\Omega,\RR^3)$.  It is clear from the assumptions on  $Q_\eps$ that $\myrho_{\eps }  \in W^{1,\infty}(\Omega, \RR^3)$.  %
Now, $Q_\eps \in \mcU$ implies $Q_\eps|_{\partial \Omega} = Q[n]|_{\partial \Omega}$;  also, since $n_\eps$ is the principal eigenvector of $Q_\eps$, $n_\eps|_{\partial \Omega}=n|_{\partial \Omega}$; overall, $\myrho_\eps|_{\partial \Omega} = 0$.

Finally, since $Q_\eps$ and  $R_{n_\eps}$ approach $Q[n]$ and $R_{n}$ with respect to their $W^{1,p}$-norms, as well as uniformly, it follows from \eqref{eq: parameterisation } that $\eps^2\myrho_{\eps } \rightarrow 0$ in $W_0^{1,p}(\Omega,\RR^3)$ as well as uniformly.
\end{proof}

\subsubsection{Strong minimality of $\Sphere^2$-valued harmonic maps}
we will require a lower bound on the  Dirichlet energy of  $\Sphere^2$-valued maps sufficiently close to a minimising harmonic map.  
The following is based on results from \cite{INSZ_new}, and is of independent interest; for completeness we give an account here.  
Let $n_b \in C^1(\partial \Omega, \Sphere^2)$ denote planar boundary conditions of nonzero degree, and let
\[ \mcN_0 := \left \{ n \in H^1(\Omega,\Sphere^2) \, : \, n|_{\partial \Omega} = n_b\right\}.\]

\begin{lemma}\label{lem: minimality of n_0}
Let $\{n_\eps\}_{\e>0} \subseteq \mcN_0$ and suppose $n_\eps \rightarrow n_0$ in {$H_0^1(\Omega,\Sphere^2)$}, where we denote by $n_0\in \mcN_0$ a minimiser of the Dirichlet energy.  There exists $\alpha > 0$ such that for all sufficiently small $\eps$, 
\begin{equation}\label{eq:claimcoercivitydif}
 \int_\Omega |\nabla n_\eps|^2 - |\nabla n_0|^2 \geqslant \alpha^2 \|n_\eps - n_0\|^2_{H_0^1(\Omega,\RR^3)}\, . 
\end{equation} 
\end{lemma}

\begin{proof} 

We first note that as $-\Delta n_0=|\nabla n_0 |^2 n_0$ and $|n_\eps|=|n_0|=1$ in $\Omega$, one has
\begin{equation}\label{eq:HTforn02}
\int_\Omega |\nabla n_\eps|- |\nabla n_0|^2 = \int_\Omega |\nabla(n_\eps - n_0)|^2 -|\nabla n_0|^2 |n_\e -n_0|^2.
\end{equation}
Now, we consider the second-order variation of the unconstrained Dirichlet energy, namely, the functional $\W:H_0^1(\Omega,\RR^3)\to \RR$ defined by 
\begin{equation}\label{eq: F defined}
\W[v]:=\int_\Omega |\nabla v|^2-|\nabla n_0|^2 |v|^2.
\end{equation}
We will reason as in \cite{INSZ_GL} to show that $\W[n_\eps-n_0]\geqslant 0$ and then use an argument inspired by one in  \cite{INSZ_new} to obtain the coercivity of this functional, which together with \eqref{eq:HTforn02} will establish the result \eqref{eq:claimcoercivitydif}.  

Since $n_0$ is a harmonic map, we have
\begin{equation}\label{eq:Phiharmonic}
-\Delta (n_0\cdot e_3)=|\nabla n_0|^2 (n_0\cdot e_3).
\end{equation} 
Also, due to Remark~\ref{rem: unique minimiser Q_0}, without loss of generality, we may assume that $n_0\cdot e_3 > 0$ in $\Omega$. This means that 
any $ \varphi\in C_c^\infty(\Omega, \RR^3)$ can be written in the form $\varphi=(n_0\cdot e_3)w$ for some $ w\in H_0^1(\Omega,\RR^3)\cap L^\infty(\Omega,\RR^3)$; just set $w:=(n_0\cdot e_3)^{-1}\varphi$. Then, using \eqref{eq:Phiharmonic} and an integration by parts we get
\be
\W[v]=\int_\Omega (n_{0}\cdot e_3)^2  |\nabla { w}|^2 \,\geqslant 0 .
\ee
The last inequality shows in particular that 
\be\label{rel:F0}
\W[v]=0 \quad \textrm{ if, and only if, } \quad \exists \gamma\in\RR^3
:\, v =(n_0\cdot e_3) \gamma \textrm{ \; in } \Omega. 
\ee

Next, consider the following constrained minimisation problem:
\begin{equation}
\lambda_1:=\inf_{v\in H^1_0(\Omega,\RR^3)}\left\{\W[v]\, :\, \|v\|_{L^2(\Omega)}=1, -|v|\leqslant 2 v\cdot n_0\leqslant 0\textrm{  in }\Omega\right\}.
\end{equation}
Standard arguments show that $\lambda_1$ is achieved by some $ v_*\in H^1_0(\Omega,\RR^3)$ with $\|v_*\|^2_{L^2(\Omega)}=1$. We claim that $\lambda_1>0$. Indeed, assume for contradiction that $\lambda_1=\W(v_*)=0$. Then, from \eqref{rel:F0}, we get $v_*= (n_0\cdot e_3) \gamma$  for some fixed $\gamma \in \RR^3$, so that the constraint $v_*\cdot n_0\leqslant 0$ reads as
\begin{equation} \label{eq:identalpha}
\gamma\cdot n_0\leqslant 0.
\end{equation} 
On the other hand, the boundary data $n_b$ has nonzero degree, and therefore for any  $e \in\Sphere^1 \times\{0\}\subset\Sphere^2$, there exists a sequence $(x_j)_{j\in\NN}$ in $\Omega$ such that $x_j\to x_b\in\partial\Omega$ and $n_0(x_j)\to e$. Hence, from \eqref{eq:identalpha}, $\gamma\cdot e \leqslant 0$ for every $e\in \Sphere^1\times\{0\}$.  Taking this into account as well as the fact that $(n_{0}\cdot e_3) > 0$ in $\Omega$, we 
must have  $\gamma=-r e_3$ for some positive $r$. But then, 
the condition $-|v|\leqslant 2 v\cdot n_0$ implies that $0<n_{0}\cdot e_3\leqslant  \frac12 $ in $\Omega$, and  this cannot happen because, otherwise, since $n_0$ is continuous, we would contradict the assumption that $n_0|_{\partial\Omega}=n_b$ has nonzero degree.
Thus,  we obtain  
\be
\int_\Omega |\nabla v|^2-|\nabla n_0|^2 |v|^2 \geqslant \lambda_1\int_\Omega |v|^2\, 
\ee
with  $\lambda_1 > 0$, provided $v \in H^1_0(\Omega, \RR^3)$ satisfies the inequality constraint $-|v| \leqslant  2 \varphi\cdot n_0 \leqslant 0$.
This implies that 
\be
\int_\Omega |\nabla v|^2-|\nabla n_0|^2 |v|^2 \geqslant \beta \int_\Omega |\nabla n_0|^2 |v|^2\, 
\ee
where $\beta = \lambda_1/ \|\nabla n_0\|^2_{L^\infty(\Omega)}> 0$ (we recall that $n_0$ is smooth), and thereby
\be
\frac{1}{1+\beta} \int_\Omega |\nabla v|^2 \geqslant \int_\Omega |\nabla n_0|^2|v|^2\, .\ee
Substituting the preceding into \eqref{eq: F defined}, we get that
\be \W[v] \geqslant \frac{\beta}{1+\beta} \int_\Omega |\nabla v|^2. \ee
The claimed relation \eqref{eq:claimcoercivitydif}  follows on setting, $\alpha^2:=\frac{\beta}{1+\beta}$, $v := n_\eps - n_0$, and noting that the inequality constraint is satisfied for all sufficiently small $\eps$.
\end{proof}

\subsubsection{Convergence estimates} \label{subsubsec: parameterisation}
The expression \eqref{eq:Ge} of the energy $\mcG_{\e}$ reads, in extended form, as
\begin{equation}\label{eq:DirichplusbulkGeps}
\mcG_\e[Q_\e] = \frac{1}{2\e^2} \int_\Omega  |\nabla Q_\eps|^2 -  |\nabla Q_0|^2 + \frac{1}{\e^4} \int_\Omega  \tilde f_{bulk} (Q_\e).
\end{equation}
We consider  separately the difference in the Dirichlet  and bulk potential energies of $Q_\e$ and $Q_0$. We first focus on the bulk energy and derive an equivalent expression of the bulk potential in terms of a suitable quadratic form. Precisely, let $\{Q_\eps\}_{\eps>0}\subseteq\mcU$ be an {\pseudomin} family. According to Lemma~\ref{prop: parameterisation}, there exist $n_\e\in W^{1,\infty}(\Omega, \Sphere^2)$, $\myrho_\e \in W_0^{1,\infty}(\Omega, \RR^3)$,  such that 
\begin{equation}\label{eq:representQeps}
Q_\eps =  Q[n_\e] + \eps^2 P^\nor_\e,\quad P^\nor_\eps :=R_{n_\eps} \Vc{\sum_{j=1}^3 \myrho_{\eps j} F_j}{\myrho_\e} R_{n_\eps}^t. 
\end{equation}
Hence, $Q_\eps = R_{n_\eps} \left(V_+ + \eps^2 
\Vc{\sum_{j=1}^3 \myrho_{\eps j}F_j}{\myrho_\e}
\right)R^t_{n_\eps}$, with $V_+ := s_+ \left(e_3 \otimes e_3 - \third I\right) \in \mcS_*$. From the rotational invariance of $\tilde f_{bulk}$ it follows that $\tilde f_{bulk}(Q_\eps) = \tilde f_{bulk}\left( V_+  + \eps^2 \Vc{\sum_{j=1}^3 \myrho_{\eps j}F_j}{\myrho_\e} \right)$.
A straightforward calculation yields
\be \label{eq: f expanded}
\tilde f_{bulk}(Q_\eps) =   \frac{\eps^4}{2} B_\eps \myrho_\eps\cdot \myrho_\eps \, ,
\ee
where $B_\eps := B_0 + \eps^2 (\myrho_{\eps}\cdot e_3) B_1 + \eps^4 |\myrho_\eps|^2 B_2$, with $B_0 = \mathrm{diag}(\mu,\mu,\nu)$ given by \eqref{eq:mcH_0 defn}, and
\begin{equation}
 B_1 = \textstyle \sqrt{\frac83} s_+ c^2  I +  \sqrt{\frac23} b^2 \, \mathrm{diag}\left(1,1,\frac{1}{3}\right), \quad B_2 = \frac{c^2}{2} I. \label{eq: Bj}
\end{equation}
%
We note that $\mu$ and $\nu$ are the coefficients of the  second  variation of $\tilde f_{bulk}$ about its minimum due to biaxial and uniaxial perturbations respectively. 
Moreover, from Lemma~\ref{prop: parameterisation}, it follows that $B_\eps \rightarrow B_0$ uniformly.  Since $B_0 = \diag(\mu, \mu, \nu)$ is positive definite, it  follows that $B_\eps$ is  positive definite  for sufficiently small  $\e$.


Next, we plug the representation of $Q_\e$ given by \eqref{eq:representQeps} into the Dirichlet part of $\mcG_\e$ (cf.~\eqref{eq:DirichplusbulkGeps}), and we expand the energy. In doing this, we note that $P^\nor_\e $ is in $H^1_0 (\Omega, \mcS_0)$ because $n_\eps$ and $n_0$ coincide on $\partial \Omega$. After a simple calculation we obtain the identity
\begin{align} \label{eq: S +T}
\frac{1}{2}\int_\Omega  |\nabla Q_\eps|^2 - |\nabla Q_0|^2  &= s_+^2 \int_\Omega |\nabla n_\eps|^2 - |\nabla n_0|^2  \nonumber   \\ 
& \qquad  +s_+ \eps^2 \int_\Omega \nabla (n_\eps \otimes n_\eps) : \nabla P^\nor_\eps+ \frac{\e^4}{2} \int_\Omega |\nabla P^\nor_\e|^2,
\end{align}
Next, recalling that $P^\nor_\e=R_{n_\eps} \Vc{\sum_{j=1}^3 \myrho_{\eps j} F_j}{\myrho_\e} R_{n_\eps}^t$ is symmetric, we get
\begin{align}\label{eq: eq eq}
 \frac{1}{2} \int_{\Omega} \nabla (n_{\varepsilon} \otimes n_{\varepsilon})
: \nabla P_{\varepsilon}^{\bot} &=  
\sum_{i = 1}^2 \int_{\Omega} \partial_i n_{\varepsilon} \otimes
n_{\varepsilon} : \partial_i P \\
 &= \sum_{i = 1}^2 \int_{\Omega} [\partial_i (P_{\varepsilon}^{\bot}
n_{\varepsilon}) - P_{\varepsilon}^{\bot} \partial_i n_{\varepsilon}] \cdot
\partial_i n_{\varepsilon} \\
&= - \int_{\Omega}\nabla n_{\varepsilon} \otimes \nabla n_{\varepsilon} :
P_{\varepsilon}^{\bot} + \sqrt{2 / 3}\int_{\Omega}  (\rho_{\varepsilon} \cdot e_3) | \nabla
n_{\varepsilon} |^2 ,
\end{align}
the last equality being a consequence of the fact that $n_{\eps}$ is an eigenvector of $P^\nor_\e$, and of the constraint $|n_\e| = 1$. 
Eventually, introducing the vector-valued function  $b_\eps \in L^\infty(\Omega, \RR^3)$ defined by
\begin{align} \label{eq:beps}
b_\e\cdot e_j &:=-2 s_+ \nabla n_\eps \otimes \nabla n_\eps :  R_{n_\eps} F_j R_{n_\eps}^t,  \ \ j=1,2,  \\ 
b_\e\cdot e_3&:= \sqrt{6} s_+ |\nabla n_\eps|^2, \label{eq:beps1}
\end{align}
we get $-s_+  \Delta (n_\eps \otimes n_\eps) : P^\nor_\eps = b_\eps \cdot \myrho_\eps $.
Overall, the energy $\mcG_\e$ can be decomposed in the form
\begin{align} \label{eq:energyrepresentation}
\mcG_\e[Q_\e] &= \frac{s_+^2}{\e^2} \left( \|\nabla n_\eps\|^2_{L^2} - \|\nabla n_0\|^2_{L^2}\right)  + \mcH_\e[n_\e,\myrho_\e]  + \frac{1}{2} \eps^2 \|\grad P^\nor_\e\|^2_{L^2} \, ,
\end{align} 
with
\begin{align} \label{eq:energyrepresentationHeps}
\mcH_\e[n_\e,\myrho_\e] &:=\int_\Omega \frac{1}{2}   B_\eps \myrho_\eps\cdot \myrho_\eps +  b_\e \cdot \myrho_\e  .
\end{align} 
Combining the above representation \eqref{eq:energyrepresentation} with Lemma~\ref{lem: minimality of n_0}, 
we obtain 
\begin{align} \label{LBin}
\mcG_\e[Q_\e] &\geqslant \frac{1}{\e^2}  {\alpha^2 s_+^2}  \|n_\eps - n_0\|^2_{H_0^1}  + \mcH_\e[n_\e,\myrho_\e]  + \frac{1}{2} \eps^2 \|\grad P^\nor_\e\|^2_{L^2}  \\ 
&\geqslant  \frac{1}{\e^2}  {\alpha^2 s_+^2}  \|n_\eps - n_0\|^2_{H_0^1} + \beta\| \myrho_\e \|^2_{L^2} - \gamma\| b_\e \|^2_{L^2} + \half \eps^2 \|\grad P^\nor_\e\|^2_{L^2} ,
\end{align} 
for some $\beta, \gamma>0$ independent of $\e$. Next, Lemma~\ref{prop: parameterisation}, assures that $\| b_\e \|_{L^2}$ is bounded independently of $\eps$ and,  therefore, the bound $\mcG_\e[Q_\e] \leqslant C$ implies that $\| n_\e - n_0 \|_{H_0^1} \leqslant C \e$, that is, \eqref{eq: n_eps - n_0 estimate}.
On the other hand,  Lemma~\ref{prop: parameterisation} also assures that $\|\myrho_\eps\|_{L^2}$  is bounded  independently of $\eps$, so there exists $\myrho \in L^2(\Omega, \RR^3)$ such that
\be \label{eq: weak eta limit lim inf}
\myrho_\eps  \rightharpoonup \myrho \quad \text{weakly in } L^2(\Omega, \RR^3).\ee
This, used in \eqref{eq:representQeps}, implies \eqref{eq: normal component of P weak}. This concludes the proof of part (\emph{ii}) of Theorem~\ref{thm:gexp2}.

\subsection{Proof of (\emph{iii}): lower bound and the existence of recovery sequences} \label{buildcutoff}
We note that \eqref{LBin} holds for any {\pseudomin} family $Q_\eps =  Q[n_\e] + \eps^2 P^\nor_\e$ having $n_\e$ for principal eigenvectors. After that, taking into account that $B_0$ is positive definite, by standard lower semicontinuity arguments we get  
\be \label{eq: weak limit inequality}
\liminf_{\eps\to 0} \mcG_\e [Q_\e] \geqslant 
\liminf_{\eps\to 0} \mcH_\e[n_\e,\myrho_\e]\geqslant  \mcH_\e[n_0,\myrho]=\int_\Omega \half B_0 \myrho \cdot \myrho + b_0 \cdot \myrho,
\ee 
where $B_0$ and $b_0$ are given by \eqref{eq:mcH_0 defn}.

To proceed, we observe that since $\Omega$ is a Lipschitz domain, it admits a family of {Hopf} cutoff functions~\cite{hopf1955nonlinear}, i.e., compactly supported smooth functions $\chi_{\e} \in C_0^{\infty} (\Omega)$ such that, for any sufficiently small $\e >0$, we have: $\chi_\e(x) =1$ if $d (x, \partial \Omega) \geqslant \e$, $ \chi_{\e} \rightarrow 1$ strongly in $L^2 (\Omega)$, and $ \| \nabla \chi_{\e} \|_{L^{\infty} (\Omega)} \leqslant C \e^{- 1}$ for some positive constant $C>0$ independent of $\e$.
Then we define, for any  $ \myrho \in L^2(\Omega, \RR^3)$,
\be
 Q_\e = Q_0 + \e^2  P^\nor_\e,\quad  P^\nor_\e =   R_{n_0}V_{\myrho_\e} R_{n_0}^t,
\ee
where $\myrho_{\e} = \chi_\e \zeta_\e$, $\zeta_\e \in C^\infty(\Omega,\RR^3)$ is such that $\zeta_\e \to \myrho$ in $L^2(\Omega,\RR^3)$, and $\|\nabla \zeta_\e\|^2_{L^2(\Omega)} \leqslant C \e^{-1} $.
The convergence relations \eqref{eq: n_eps - n_0 estimate}, \eqref{eq: normal component of P weak} are trivially satisfied because for any $\e>0$ the director $n_\e$ is the principal eigenvector of $Q_\e$. In particular, a direct computation yields
\bea \label{eq:recoverysequence}
\lim_{\e\to 0}\mcG_\e[Q_\e] &=   \lim_{\e\to 0} \left( \int_\Omega  \frac12 \left( B_\eps \myrho \cdot \myrho +   b_\e \cdot  \myrho\right) \chi_\e  + \frac{\e^2}{2} \int_\Omega |\nabla  P^\nor_\e|^2 \right) .
\eea
Denoting by $\Omega_{\e} \assign \{ x \in \Omega : d (x,
  \partial \Omega) < \e \}$ the tubular neighbourhood of $\partial
  \Omega$ of radius $\e$, we obtain for $\e$ sufficiently small, the existence of a positive constant $C_0$ depending only on $n_0$ such that
  \begin{equation}
    \frac{\e^2}{2} \int_{\Omega} \left| \grad  P_\eps^\nor \right|^2
    \leqslant  C_0 \left( \e^2 \int_{\Omega_{\e}} \left| \zeta_\e \grad \chi_{\e}  \right|^2+{\e^2\int_\Omega  \left| \chi_{\e} \grad \zeta_\e \right|^2 +|\chi_\e  \zeta_\e |^2} \right) \to 0 . \label{eq:tempestimateHOPF}
  \end{equation} 
Combining the previous estimate with \eqref{eq:recoverysequence}, and recalling the definition of $B_\e$ and $\xi_\e$, we infer that 
\begin{equation} \label{eq:gammalimsup}
\lim_{\e\to 0} \mcG_{\e} [Q_{\e}] =  \mcH_0 [n_0, \myrho].
\end{equation}
This establishes (\emph{iii}) of Theorem~\ref{thm:gexp2}.

\subsection{Proof of statement (\emph{iv}): $\Gamma$- convergence and convergence estimates for the  minimisers}
The $\Gamma$-convergence of $\mcG_\e$ to $\mcH_0[n_0,\myrho_0]$, with $\myrho_0 := { -B_0^{-1}} \, b_0$, is clear from the lower bound \eqref{eq: weak limit inequality} and the upper bound \eqref{eq:gammalimsup}. It remains to prove the convergence estimates for the minimisers.
Let $\{ Q^*_\e \}_{\e>0} \subseteq W^{1, \infty} (\Omega, \mcS_0)$ be a family of minimisers of $\mcE_\eps$.
According to Lemma~\ref{prop: parameterisation}, $Q^*_\e$ may be expressed in terms of its principal eigenvector, $n^*_\e\in W^{1, \infty} (\Omega, \Sphere^2)$, and the vector-valued function $\myrho^*_\e \in W_0^{1,\infty}(\Omega,\RR^3)$.
Precisely, we have
\begin{equation}
\label{eq: parameterisationminimizer }
Q^*_\e  =  Q[n^*_\e] + \eps^2 P^{*\nor}_\eps,\quad
P^{*\nor}_\eps :=R_{n^*_\eps} \Vc{\sum_{j=1}^3 \myrho^*_{\eps j} F_j}{\myrho^*_\e} R_{n^*_\eps}^t,
\end{equation}
with $\e^2 \myrho^*_\e\to 0$ in $W^{1,p}_0(\Omega,\RR^3)$.
Since $\mcG_\e[Q^*_\e ]$ is bounded, it follows from the same argument that led to~\eqref{eq: weak eta limit lim inf}, that, perhaps up to a subsequence,  $\myrho^*_\e$ converges weakly in $L^2(\Omega, \RR^3)$ to some $\myrho^{*}$. In particular, we have
\begin{align}
B_\e & \to B_0\quad \text{strongly in }L^2(\Omega,\RR^{3\times 3})\,,\\
\myrho_{\e}^{*} & \rightharpoonup \myrho^{*} \quad \text{weakly in }L^2(\Omega,\RR^3),
\end{align}
where $B_\e \assign B_0 + \eps^2(\myrho_{\eps}^{*} \cdot e_3) B_1 + \eps^4 | \myrho_{\eps}^{*} |^2 B_2$.
Since $B_0$ is positive definite, by the lower semicontinuity of the norms and \eqref{eq:energyrepresentation}, we
have that
\begin{equation}
  \label{eq: weak limit inequality 2} \liminf_{\eps \to 0} \mcG_{\e}
  [Q^{*}_{\e}] \geqslant \liminf_{\eps \to 0} \mcH_{\e}
  [n_{\e}^{*} , \myrho_{\e}^{*} ] \geqslant \mcH_0 [n_0, \myrho^{*} ]
  = \int_{\Omega} \frac{1}{2} B_0 \myrho^{*}  \cdot \myrho^{*}  + b_0\cdot \myrho^{*}  \geqslant \mycal{H}_0 [n_0, \myrho_0],
\end{equation}
with $\myrho_0 \assign \mathrm{argmin}_{\sigma \in \RR^3} (\frac{1}{2} B_0
\sigma \cdot \sigma + b_0 \cdot \sigma) = - B_0^{- 1} b_0$, and $b_0$ given by \eqref{eq:mcH_0 defnb}. Also, by (\emph{iii}), there exists
an {\pseudomin} recovery  family $\{Q_{\e} \}_{\e > 0} \subseteq \mcU$
such that $\lim_{\e \to 0} \mcG_{\e} [Q_{\e}] =\mcH_0 [n_0,
\myrho_0]$. Since $\mcE_{\e} [Q_{\e}] \geqslant \mcE_{\e}
[Q^{*}_{\e}]$, it follows that  $\lim_{\e \to 0} \mycal{G}_{\e}
  [Q^{*}_{\e}]=\mcH_0 [n_0, \myrho_0]$ because
\begin{equation}
  \label{eq: G} \mcH_0 [n_0, \myrho_0] = \lim_{\e \to 0}
  \mcG_{\e} [Q_{\e}] \geqslant \limsup_{\e \to 0} \mycal{G}_{\e}
  [Q^{*}_{\e}] \geqslant \liminf_{\e \to 0} \mcG_{\e}
  [Q^{*}_{\e}] \geqslant \mcH_0 [n_0, \myrho_0] .
\end{equation}
From \eqref{eq:energyrepresentation} and the preceding, we deduce that
\begin{eqnarray}
  \mcG_{\e} [Q_{\e}^{*}] -\mcH_0 [n_0, \myrho_0] & = &
  \frac{s_+^2}{\e^2}  \left( \| \nabla n_{\eps}^{*}  \|^2_{L^2} -\| \nabla
  n_0 \|^2_{L^2} \right) + \frac{1}{2} \eps^2 \| \grad P^{\nor}_{\e}
  \|^2_{L^2} \nonumber\\
  &  &  \quad \quad \quad \quad \quad \quad \quad
  +\mcH_{\e} [n_{\eps}^{*} , \myrho_{\e}^{*} ] -\mcH_0
  [n_0, \myrho_0] .  \label{eq:jdf1}
\end{eqnarray}
On the other hand, since $\rho_0 = - B_0^{- 1} b_0$, we have
\begin{eqnarray}
  \mcH_{\e} [n_{\eps}^{*} , \myrho_{\e}^{*} ] -\mcH_0
  [n_0, \myrho_0] & = & \int_{\Omega} \frac{1}{2} B_{\e}
   \myrho_{\e}^{*}  \cdot \myrho_{\e}^{*}  + b_\e
   \cdot \myrho_{\e}^{*}  - \int_{\Omega}
  \frac{1}{2} B_0 \myrho_0 \cdot \myrho_0 + b_0 \cdot \myrho_0 \nonumber\\
  & = & \frac{1}{2} \int_{\Omega} B_0^{- 1} b_0 \cdot b_0 - B_\e^{- 1} b_\e^{*} \cdot b_\e^{*} \nonumber \\
  &  & \qquad\qquad+ \frac{1}{2}  \int_{\Omega} B_\e
  \left( \myrho_{\e}^{*}  + B_\e^{- 1} b_\e^{*} \right) \cdot \left( \myrho_{\e}^{*}  + B_{\e}^{- 1} b_\e^{*} \right) . \nonumber
\end{eqnarray}
with $b_\eps^{*}$ defined as in \eqref{eq:beps}, \eqref{eq:beps1}.  Since $n_{\varepsilon} \rightarrow n_0$ strongly in $W^{1, p} (\Omega,
\Sphere^2)$, it follows that $b_{\varepsilon}^{*} \rightarrow b_0$ strongly in
$L^2 (\Omega, \RR^3)$. Hence,
\[ 0 \leqslant \frac{1}{2}  \int_{\Omega} B_0 \left( \myrho^{*} + B_0^{- 1} b_0
   \right) \cdot \left( \myrho_0 + B_0^{- 1} b_0 \right) \leqslant
   \liminf_{\varepsilon \rightarrow 0} \left( \mycal{H}_{\e} [n_{\eps}^{*},
   \myrho_{\e}^{*}] -\mycal{H}_0 [n_0, \myrho_0] \right) . 
   \]
Summarizing, from the previous inequality and \eqref{eq:jdf1}, we infer that
\begin{eqnarray}
  0 \geqslant \lim_{\varepsilon \rightarrow 0} \left( \mycal{G}_{\e}
  [Q_{\e}^{*}] -\mycal{H}_0 [n_0, \myrho_0] \right) & \geqslant &
  \lim_{\varepsilon \rightarrow 0} \left[ \frac{s_+^2}{\e^2}  \left( \| \nabla
  n_{\eps}^{*} \|^2_{L^2} -\| \nabla n_0 \|^2_{L^2} \right) + \frac{1}{2}
  \eps^2 \| \grad P^{\nor}_{\e} \|^2_{L^2} \right] \nonumber\\
  &  &  \quad \quad \quad \quad + \frac{1}{2}  \int_{\Omega}
  B_0 \left( \myrho^{*} + B_0^{- 1} b_0 \right) \cdot \left( \myrho_0 + B_0^{- 1}
  b_0 \right) . \nonumber
\end{eqnarray}
As each term on the right-hand side is nonnegative, they separately
vanish in the limit $\e \to 0$. In particular, $\myrho^{*} = - B_0^{- 1} b_0 = \myrho_0$ and, by Lemma~\ref{lem: minimality of n_0}, $(n_{\e}^{*} - n_0) / \e \to 0$ in
$H^1 (\Omega, \Sphere^2)$. This
establishes the convergence estimates \eqref{eq: n_e - n_0 = o(eps) thm 2.1} and
\eqref{eq: normal component minimiser}, and completes the proof of Theorem~\ref{thm:gexp2}.

\section{The case $b^2 = 0$ and non-orientable boundary conditions.}\label{sec: b = 0 results}
Generally speaking, for non-orientable boundary conditions on a two-dimensional domain, the Landau-de Gennes energy $\mcE_\eps[Q_\eps]$ of a minimising sequence $Q_\eps$ diverges logarithmically as $\eps \rightarrow 0$ (cf.~\cite{Canevari2017line}), and an analysis different from the one developed in this paper is required to describe the small-$\eps$ behaviour.  However, in the special case  $b^2 = 0$ in the Landau-de Gennes bulk potential, results similar to those of Section~\ref{sec:math} can be established.  The key point is that $b^2 = 0$ corresponds to a degeneracy in the bulk potential, which reduces to a function of  $\tr\, Q^2$ only, 
\begin{equation}
\label{eq: b = 0 bulk potential }
\tilde f_{bulk} = \left( \frac{c}{
  4\e} \right)^2 \left( | Q |^2 - \frac{a^2}{c^2} \right)^2,
\end{equation} 
with  four-dimensional limit manifold
\be \label{eq:limmanifold4}
\mcS_0^*:=\{Q\in \mcS_0; |Q|^2= (2/3) s_+^2=a^2/c^2\}
\ee
 homeomorphic to $\Sphere^4$, as opposed to $\RR P^2$ in the generic case.  

In addition to taking boundary conditions to lie in the degenerate limit manifold $\mcS_0^*$, we restrict them to be planar prolate uniaxial, in analogy with the $b^2 \neq 0$ case.  This allows for a convenient generalisation of degree to non-orientable boundary conditions, as follows.   Let $\mcB_Q$ denote the set of planar prolate uniaxial $Q$-tensors in $\mcS_0^*$, and $\mcB_D = \{n \in \Sphere^2 \, | \, n\cdot e_3 = 0\}$ denote the set of planar directors.  The parameterisation $n \mapsto Q = s_+(n\otimes n - \third I)$ is a double covering of $\mcB_Q$ by $\mcB_D$ (since $n$ and $-n$ parameterise the same $Q$-tensor).  Since $\mcB_D$ is homeomorphic to $\Sphere^1$, it follows that $\mcB_Q$ is homeomorphic to the real projective line $\Sphere^1/\ZZ_2$, which is also homeomorphic to $\Sphere^1$ via the map
\begin{equation}
e^{i\theta}\in  \Sphere^1 \mapsto \{e^{i\theta/2},e^{i(\pi + \theta/2)}\}\in \Sphere^1/\ZZ_2.
\end{equation}
Thus, boundary conditions $Q_b \in C^1(\partial \Omega, \mcB_Q)$ may be assigned an integer degree, $\tdeg Q_b$.  If 
$\tdeg Q_b$ is even, say equal to $2m$, there exists a planar director $n_b \in  C^1(\partial \Omega, \mcB_D)$ such that $\deg n_b = m$; in this case, $Q_b$ is orientable.  In the non-orientable case, $\deg Q_b = k$ is odd; 
any $n_b$ which parameterises $Q_b$ necessarily has a discontinuity in sign, so that $n_b \notin  C^1(\partial \Omega, \mcB_D)$.  In this case, one says that $n_b$ has half-integer degree $k/2$.

 Throughout this section we assume that $\Omega$ is a bounded, simply-connected domain with $C^1$ boundary.
%
The following result can be shown in a manner similar  to that of Proposition~\ref{prop:zero}.
\begin{theorem} \label{prop:zero0 b=0}
Let $Q_b \in C^1(\partial\Omega, \mcB_Q)$  and let  $\mcU = \{Q \in H^1(\Omega, \mcS_0) \, ; \, Q|_{\partial \Omega} = Q_b\}$. 
Then, as $\e \to 0$, the following statements hold:
\begin{enumerate}[leftmargin=0.8cm,label=\textup{(}\roman*\textup{)},widest = i]
\item Let $C>0$. For any family $\{Q_\e\}_{\e>0} \subset \mcU$ such that $\mcE_\e[Q_\e] \leqslant C$ we have, possibly on a subfamily, $Q_\e \to Q$ weakly in $H^1(\Omega, \mcS_0)$ for some  $Q \in H^1(\Omega, \mcS_0^*)$.
\medskip
\item  The family of energies $(\mcE_\e)_{\e>0}$ $\Gamma$-converges to $\mathcal E_0$ in the weak topology of $H^1 (\Omega, \mcS_0)$, where
\begin{equation}
\label{eq: Gamma0}
 \mcE_0 [Q] = \begin{cases} 
\displaystyle \half \int_{\Omega}  |\nabla Q|^2& \text{ if $Q \in H^1(\Omega, \mcS_0^*)\cap \mcU$},\\
+\infty &\text{otherwise},
\end{cases}
\end{equation} 
with $\mcS_0^*$ the limit manifold defined by \eqref{eq:limmanifold4}.
\medskip\item The minimisers $\{Q_\e\}_{\e>0}$ of the problem \eqref{main_problem} converge strongly in $H^1(\Omega, \mcS_0)$ to the minimisers of the following harmonic map problem 
\be \label{eq:problemH01}
 \min_{Q\in\mcU}\mathcal E_0 [Q]. 
\ee 
\end{enumerate}
\end{theorem}

%

\begin{remark}\label{rem: unique minimiser Q_0 b = 0}
Note that a planar uniaxial $Q$-tensor  $Q = s_+(n \otimes n - \third I)$ has the following expression in terms of the orthonormal basis \eqref{eq:Fbasis}:
\begin{equation}
\label{eq: uniaxial Q-tensor expansion }
Q = \sum_{j=1}^5 c_j F_j = \frac{s_+}{\sqrt 2}\left( (n_1^2 - n_2^2) F_1 + 2n_1 n_2 F_2 - \frac{1}{\sqrt 3} F_3\right),
\end{equation}
where $n_i = n\cdot e_i$, $i  = 1,2$.  
Thus, $c_4$ and $c_5$ vanish, while $c_3$ is fixed and negative. It follows that every element $Q_b \in C^1(\partial\Omega, \mcB_Q)$ admits a representation of the form $Q_b = \sqrt{2/3} s_+ \sum_{j=1}^{3} c_{bj} F_j$, for some vector field $c_b\in C^1(\partial\Omega, \Sphere^2)$. After that, standard arguments based on the maximum principle show the existence of a unique minimiser of problem \eqref{eq:problemH01}; it can be expressed as
\be \label{eq: form of Q_0 for b = 0}
Q_0 = \sqrt{2/3} s_+ \left(c_{01} F_1 + c_{02} F_2 + c_{03} F_3\right),
\ee
where $c_0 \in H^1(\Omega, \Sphere^2)$ 
solves the following minimisation problem:
\be \label{eq: min problem for c_0}
\min \left\{ \int_\Omega | \nabla c\,|^2 \, : \, {c \in H^1(\Omega, \Sphere^2),\  c= c_b \hbox{ on } \partial \Omega}\right\}.
\ee
In particular, $c_0$ is an $\Sphere^2$-valued harmonic map, i.e., $-\Delta c_0 = |\nabla c_0|^2 c_0$.  
We note that $Q_0$ is biaxial unless one of the following conditions holds: \emph{i}) $c_0 \cdot e_3 = -1/\sqrt{3}$, in which case $Q_0$ is planar uniaxial; \emph{ii}) $c_0 \cdot e_3 = -1$, in which case $Q_0 = -\sqrt{2/3} s_+  F_3$ is oblate uniaxial with director $e_3$; or \emph{iii}) $c_0 \cdot e_3 = 1$, in which case $Q_0 = \sqrt{2/3} s_+  F_3$ is prolate uniaxial with director $e_3$.  In fact, the maximum principle implies that $c_0 \cdot e_3 < 0$ 
, so that the last possibility is excluded.
\end{remark}
We need to go to the next-order term in the $\Gamma$-asymptotic expansion of
the energy $\mcE_\e$ and define the 
 renormalised relative energy
as in \eqref{eq:Ge},
\begin{equation}
 \mcG_{\e} [Q] = \frac{1}{\e^2} \left( \mcE_\e[Q] - \mcE_0 [Q_0] \right),
 \end{equation}
where $Q_0$ is the unique minimiser of the problem \eqref{eq:problemH01}; in particular,  $Q_0$ is a harmonic map.
Information about the expansion of the energy $\mcE_\e$ is given by the following result.
\begin{theorem}\label{thm:gexp}
 Let  $Q_0$ be a minimiser of $\mcE_0$ over $\mcU$ as  in the problem \eqref{eq:problemH01}. The following assertions hold:
\begin{enumerate}[leftmargin=0.8cm,label=\textup{(}\roman*\textup{)},widest = i]
\item Let $C>0$. For any family $\{Q_{\e}\}_{\e>0} \subseteq \mcU$ such that  $\mcG_{\e} [Q_{\e}] \leqslant C$, there exist $P\in H_0^1(\Omega, \mcS_0)$, pointwise orthogonal to $Q_0$, and $\myrho\in L^2(\Omega)$,
for which, possibly on a subsequence,
\begin{align}
Q_{\e} \to Q_0 & \text{\qquad strongly in $H^1(\Omega,\mcS_0)$} \label{eq:compb01}\\
\frac{1}{\e^2} \left( Q_\eps  - Q_0 \right) : Q_0   \rightharpoonup \myrho & \text{\qquad weakly  in  $L^2(\Omega)$,}\label{eq:compb02}\\
\frac{1}{\e} (Q_\e - Q_0) \rightharpoonup P & \text{\qquad weakly  in  $H_0^1(\Omega, \mcS_0)$ with $P:Q_0 =0$\label{eq:compb03}.}
\end{align}
\item  For any $\{Q_{\e}\}_{\e>0} \subseteq \mcU$ such that \eqref{eq:compb01},\eqref{eq:compb02}, and \eqref{eq:compb03} hold, we have
\be \label{eq:en0gdf1}
 \liminf_{\e\to 0}  \mcG_{\e}[Q_\e] \geqslant \mcH[P,\myrho],
\ee
with
\be \label{eq:UB0gdf}
\mcH[P,\myrho]:= \frac12 \int_{\Omega}
  | \nabla P |^2 + \int_{\Omega}\frac{c^2}{a^2} | \nabla Q_0 |^2 \myrho + \frac{c^2}{4} \int_{\Omega} (|P|^2 + 2 \myrho)^2 .
  \ee
Also, for any $P\in H_0^1(\Omega, \mcS_0)$ pointwise orthogonal to $Q_0$, and any $\myrho\in L^2(\Omega)$, there exists a recovery family $\{Q_{\e}\}_{\e>0} \subseteq \mcU$ such that \eqref{eq:compb01}, \eqref{eq:compb02}, \eqref{eq:compb03} hold, and 
\begin{equation}
\lim_{\e \rightarrow 0} \mcG_{\e} [Q_{\e}]=\mcH[P,\myrho].
\end{equation}

\medskip
\item The family of energies $\{\mcG_\e\}_{\e>0}$ $\Gamma$-converges to $\mcG_0$ in $H^1(\Omega,\mcS_0)$, where
\be \label{eq:en0}
 \mcG_0[Q]=  \begin{cases} 
\displaystyle-\frac{c^2}{4 a^4} \int_\Omega |\nabla Q_0|^4& \text{ if $Q =Q_0$},\\
+\infty &\text{otherwise}.\end{cases} 
\ee
Moreover if $(Q_\e)_{\e>0}$ is a family of minimisers of $\mcE_\e$ on  $\mcU$ then 
\begin{align}
\frac{1}{\e^2} (Q_\e -Q_0):Q_0 \to \frac{1}{2a^2} |\nabla Q_0|^2  & \text{\qquad in $L^2(\Omega, \mcS_0)$,} \label{eq:mincon0} \\
\frac{1}{\e} (Q_\e - Q_0) \to 0  & \text{\qquad in $H^1(\Omega, \mcS_0)$.}
\label{eq:mincon1}
\end{align}
\end{enumerate}
\end{theorem}
\begin{proof}
\noindent (\emph{i})  If $Q_\eps$ satisfies $\mcG_{\e} [Q_\eps] \leqslant C$, by the same argument used in the proof of the Theorem~\ref{thm:gexp2}, we get that
necessarily $Q_\e \to Q_0$ strongly in $H^1(\Omega, \mcS_0)$.
After that, let $(Q_\e)_{\e>0}$ be such that $Q_\e \to Q$ in $H^1(\Omega, \mcS_0)$. We set $P_{\e} \assign (Q_\e - Q_0) / \e^2$, so that $Q_\e = Q_0 + \e^2 P_\e$ with
$P_{\e} \in H_0^1 \left( \Omega, \mcS_0 \right)$. Plugging the expression of
$Q_\e$ into the energy $ \mcG_{\e}$, and taking into account that $Q_0$ is a harmonic map, we obtain
\begin{align}
  \mcG_{\e}[Q_\e] & = \frac{1}{2}  \int_{\Omega}
  \e^2 | \nabla P_{\e} |^2 + \int_{\Omega} \nabla Q_0:\nabla P_\eps + \frac{c^2}{4} \int_{\Omega} (\e^2 |
  P_{\e} |^2 + 2 Q_0 : P_{\e})^2 \\
  &= \frac{1}{2}  \int_{\Omega}
  \e^2 | \nabla P_{\e} |^2 + \int_{\Omega}\frac{c^2}{a^2} | \nabla Q_0 |^2 (Q_0: P_{\e}) + \frac{c^2}{4} \int_{\Omega} (\e^2 |
  P_{\e} |^2 + 2 Q_0 : P_{\e})^2,
  \label{eq:Genergy}
\end{align}
and, after some further computation,
\begin{align}
  \mcG_{\e}[Q_\e] & =  \int_\Omega \left(\frac{c}{2a^2}|\nabla Q_0|^2+c(Q_0:P_\e)+\frac{c}{2}\e^2|P_\e|^2 \right)^2 - \frac{c^2}{4 a^4} \int_{\Omega} | \nabla Q_0 |^4\nonumber\\
  &\qquad\qquad+\frac{\eps^2}{2}\int_\Omega |\nabla P_\e|^2-\frac{c^2}{a^2}|\nabla Q_0|^2 |P_\eps|^2.
  \label{eq:Genergynew}
\end{align} 
Using the decomposition trick (cf.~Lemma A.1.~in \cite{Ignat2015stability}) we claim  that, for some $\alpha>0$, the following estimate holds:
\be \label{eq:hardy}
\int_\Omega |\nabla P_\e|^2-\frac{c^2}{a^2}|\nabla Q_0|^2 |P_\eps|^2 \geqslant \alpha \int_\Omega |\nabla P_\e|^2.
\ee
 Indeed, we know that $q_3=Q_0 : F_3$ solves $-\Delta q_3 = \frac{c^2}{a^2} |\nabla Q_0|^2 q_3$ and, by the maximum principle, $q_3<0$ in $\Omega$ because $\min_{\partial\Omega} q_3<0$.  Thus, we can represent any second-order  perturbation in the form $P_\e = q_3 U_\e$ with $U_\e:=q_3^{-1}P_\e$. Arguing as in the proof of Lemma \ref{lem: minimality of n_0}, we deduce the existence of a positive constant $\beta$ such that
\begin{align}\label{eq:Htrick}
\int_\Omega |\nabla P_\e|^2-\frac{c^2}{a^2}|\nabla Q_0|^2 |P_\eps|^2 & =
\int_\Omega |\nabla q_3 U_\e + q_3 \nabla U_\e|^2 + \Delta q_3 \, q_3 \, |U_\e|^2 \notag \\
& \hspace{1cm}=\int_\Omega |q_3|^2 |\nabla U_\e|^2 \;\geqslant\; \beta \int_\Omega \frac{c^2}{a^2} |\nabla Q_0|^2 |P_\e|^2.
\end{align}
This, for $\alpha := \frac{\beta}{1+\beta}$, immediately implies the desired result \eqref{eq:hardy}.

Since $\mcG_\e (Q_\e) \leqslant C$, by \eqref{eq:Genergynew} and \eqref{eq:hardy}, we obtain $ \| \e P_\e\|_{H^1} \leqslant C$ and  $\| P_\e : Q_0\|_{L^2} \leqslant C $.  Thereby, the existence of $P\in H_0^1(\Omega, \mcS_0)$, $\myrho\in L^2(\Omega)$ such that $Q_0: P_\e \rightharpoonup \myrho$ weakly in $ L^2(\Omega)$, and $\e P_\e \rightharpoonup P $ weakly in $H_0^1(\Omega,\mcS_0)$. Therefore, also $P:Q_0 =0$.
\smallskip

\noindent (\emph{ii}) The lower bound \eqref{eq:en0gdf1} follows from \eqref{eq:Genergynew} and the lower semicontinuity of the norms under weak convergence.
 Now, for any $P\in H_0^1(\Omega, \mcS_0)$ pointwise orthogonal to $Q_0$, and any $\myrho\in L^2(\Omega)$, we want to construct a recovery family $\{Q_{\e}\}_{\e>0} \subseteq \mcU$ such that \eqref{eq:compb01}, \eqref{eq:compb02}, \eqref{eq:compb03} hold, and  $\lim_{\e \rightarrow 0} \mcG_{\e} [Q_{\e}]=\mcH[P,\myrho]$. To this end we recall the construction for the case $b\neq 0$ and define 
$\xi_{\e} = \chi_\e \zeta_\e$ with $\zeta_\e \in C^\infty(\Omega)$, $\zeta_\e \to \myrho$ in $L^2$, $\|\nabla \zeta_\e\|^2_{L^2} \leqslant C \e^{-1}$  and $\chi_\e$ defined as in section \ref{buildcutoff}. For any $P \in H_0^1(\Omega, \mcS_0)$ such that $P:Q_0 =0$ we set, as a recovery family, 
\begin{equation}
  Q_{\e} \assign Q_0 + \e P + \frac{3}{2s_+^2}\e^2 \xi_{\e} Q_0.
\end{equation}
Plugging this expression into {\eqref{eq:Genergy}} we infer
\begin{align}
   \mcG_{\e} (Q_{\e})  &=  
  \frac{1}{2}  \int_{\Omega}
  | \nabla (P + \e \xi_{\e} Q_0) |^2 + \int_{\Omega}\frac{c^2}{a^2} | \nabla Q_0 |^2 \xi_\e + \frac{c^2}{4} \int_{\Omega} \left( |
  P + \textstyle{\frac{3\e}{2s_+^2}} \xi_\e Q_0 |^2 + 2 \xi_\e\right)^2.
  \end{align}
Finally, taking the limit as $\e \to 0$ we conclude.

\smallskip
\noindent (\emph{iii})
It is clear that if $Q \neq Q_0$ we can take $Q_\e =Q$ to recover $\mcG_\e [Q_\e] \to \infty$.
It is also clear that if $Q_\e$ is a family of minimisers of $\mcE_\e$ ,then $(i)$ holds. 
Minimising \eqref{eq:en0} with respect to $P$ and $\myrho$, we obtain $P=0$ and $\myrho =\frac{1}{2a^2} |\nabla Q_0|^2$. Moreover the minimal energy is
$$
\min \mcG_0 (\myrho) = -\frac{c^2}{4 a^4} \int_\Omega |\nabla Q_0|^4.
$$
In order to obtain  \eqref{eq:mincon0} and \eqref{eq:mincon1}, we combine \eqref{eq:Genergy} with the results stated in (\emph{ii}).
\end{proof}

\section{Applications to conformal director fields}\label{sec: conformal n}

Our previous results provide  refined information on minimisers of the Landau-de Gennes energy for any fixed planar boundary conditions of nonzero degree. 
In this section we apply Theorem~\ref{thm:gexp2} and Theorem~\ref{thm:gexp}  to two  families of planar boundary conditions  $n_b$ of independent interest. In particular, we consider a class of boundary data for which $Q_0$, the leading-order  Landau-de Gennes minimiser, is, up to a normalisation factor, an $\Sphere^4$-valued harmonic map.  
 In both cases ($b=0$ and $b \neq 0$), $Q_0$ is related to a conformal (and therefore harmonic) $\Sphere^2$-valued map.  However, the relationship is different in the two cases. In the case $b^2 \neq 0$, $Q_0$ is given by $Q[n_0]$, where $n_0$ is a conformal director field.  In the case $b^2 = 0$, $Q_0$ is given up to normalisation by $c_{01} F_1 + c_{02} F_2 + c_{03} F_3$, where $c_0:\Omega \rightarrow \Sphere^2$ is conformal.
These conformal families are parameterised by the positions of interior {\it escape points}, where $n_0$ or $c$ is vertical, i.e., parallel to $e_3$.

The above class of boundary conditions is interesting for several reasons.  First, the leading-order Oseen-Frank energy saturates a topological lower bound, and is the same for all boundary conditions within the family. Therefore it is impossible to distinguish between minimal $Q$-tensor configurations generated by these boundary conditions using only the leading-order approximation.  The first-order correction breaks this degeneracy, and provides a mechanism to describe how the Landau-de Gennes energy depends on the position of escape points (defined by the boundary conditions) for $Q$-tensor fields that are harmonic at leading order.  Also, rather explicit results are available for both the leading- and next-order Landau-de Gennes minimiser in terms of the Green's function of the Laplacian on $\Omega$.  Interestingly, for these special boundary conditions, the biaxial component of the next-order correction vanishes; biaxiality appears only at order higher than $O(\eps^2)$.  
Results for the case $b^2 \neq 0$ are stated in Section~\ref{subsec: conformal - results}, and  proofs are given in Section~\ref{subsec: conformal - proofs}.
Results for the case $b^2 = 0$ are stated in Section~\ref{subsec: b = 0}.


\subsection{Harmonic $Q$-tensors and conformal director fields -- main results}\label{subsec: conformal - results}
We begin by establishing a connection between harmonic uniaxial $Q$-tensors and conformal director fields.
\begin{definition} \label{def: conformality} 
 A director field $n \in H^1(\Omega, \Sphere^2)$ is {\it conformal}  if
\begin{equation}
\label{eq: conformal director}
\partial_2 n = \sigma n \times \partial_1 n,
\end{equation}
with $\sigma \equiv 1$ or $\sigma \equiv -1$ in $\Omega$.
 \end{definition} 

\noindent If $\Omega$ is equipped with the Euclidean metric and $\Sphere^2$ equipped with its standard Riemannian metric, then \eqref{eq: conformal director} is equivalent to the usual definition of conformal maps as isometries up to a scale factor; the sign $\sigma$ determines whether $n$ is orientation-preserving $(\sigma = 1)$ or reversing $(\sigma = -1)$.  
\begin{proposition}\label{prop: conformal implies harmonic}
If $n \in H^1(\Omega, \Sphere^2)$ is conformal, then $n$ is an $\Sphere^2$-valued harmonic map.
\end{proposition}
The proof 
involves showing that $n$ conformal implies that $n$ is a weakly harmonic map.  One then appeals to a 
result of Hel\'ein \cite{Helein2d} that weakly harmonic maps over two-dimensional domains are real analytic.

A director field $n$ may be identified with a complex-valued function $w$ on $\Omega$
via stereographic projection between $\Sphere^2$ and the extended complex plane ${\mathbb C}^*$, as follows:
\begin{equation}
\label{eq: stereo}
w = \frac{n_1 + i n_2}{1+ n_3}, \qquad n = \frac{\left( 2\Re w,  2\Im w,   1-|w|^2\right)}{1+|w|^2}.
\end{equation}
Then $n$ being conformal is equivalent to $w$ being either meromorphic $(\sigma = 1)$ or antimeromorphic $(\sigma = -1)$.

We identify  $\Sphere^4$ with the space of $Q$-tensors of unit norm.
\begin{definition}
A $Q$-tensor field  $Q \in H^1(\Omega, \Sphere^4)$  is a (weakly) {\it $\Sphere^4$-valued harmonic map} if
\begin{equation}
\label{eq: harmonic Q-tensor }
 \Delta Q = -{|\nabla Q|^2} Q \text{ in } {\mycal D}'(\Omega,\mcS_0).
\end{equation}
\end{definition} 
As with director fields, if $Q$ is a weakly harmonic map, it is real analytic \cite{Helein2d}.

\begin{proposition}\label{prop: harmonic uniaxial Q-tensor}
Let  $n \in H^1(\Omega, \Sphere^2)$ and define $Q\in H^1(\Omega, \Sphere^4)$ by
\begin{equation}
\label{eq:  Q uni harm}
Q = \sqrt{3/2} \left(n \otimes n - \third I \right).
\end{equation}
Then $Q$ is an $\Sphere^4$-valued harmonic map if and only if $n$ is conformal.
\end{proposition}
\noindent The proof is given in Section~\ref{subsec: proof of harmonic Q = conformal n}.  Below,
in a slight abuse of terminology we will say

\begin{definition} A $Q$-tensor field $Q \in H^1(\Omega,\mcS_0)$ is {\it harmonic} if $|Q|$ is everywhere constant and $Q/|Q|$ is an $\Sphere^4$-valued harmonic map.
\end{definition}

Next, we use the connection between harmonic uniaxial $Q$-tensors and conformal director fields to determine the planar boundary conditions 
of given degree that minimise the leading-order Landau-de Gennes energy.
%
Given $a \in \Omega$, let $g_a \in C^\infty(\Omega)$ denote the solution of the Laplace equation 
\begin{equation}
\label{eq: g }
\Delta g_a = 0, \quad g_a|_{\partial \Omega}(x) = \log|x - a|.
\end{equation}
Thus, $\log|x - a| - g_a$ is the Green's function for the Laplacian on $\Omega$ with Dirichlet boundary conditions. 

In what follows, it will be convenient to regard $\Omega$ as a subset of $\CC$  rather than $\RR^2$; expressions such as $1/(x-a)$ for $x, a \in \Omega$ should be understood in this context.  Since $\Omega$ is simply connected, $g_a$ has a harmonic conjugate, which is determined up to an additive constant. Let $h_a$ denote a harmonic conjugate of $g_a$.    Then $g_a + i h_a$ is holomorphic on $\Omega$. 
Let  $m \in \ZZ$ and $\bm{a} = (a_{(1)}, \ldots, a_{(|m|)}) \in \Omega^{|m|}$ denote an $|m|$-tuple of points in $\Omega$, not necessarily distinct.  We define 
\begin{equation}\label{eq: w}
w_{0; \bm{a}}  :=   e^{i\alpha} \left[ \prod_{j=1}^{|m|} \frac{x-a_{(j)}}{ \exp\bigl( g_{a_{(j)}}+   ih_{a_{(j)}}\bigr)} \right]^{\sgn m}\, ,
\end{equation}
for some $\alpha\in\RR$.

%
%
%
%
%



 \begin{theorem}\label{thm:locdef}  Let $n_b \in C^1(\partial\Omega, \Sphere^1)$ be a planar boundary director field of degree  $m \neq 0$, and let $Q_b = s_+(n_b \otimes n_b - \third I)$. The following assertions hold:
 \begin{enumerate}[leftmargin=0.8cm,label=\textup{(}\roman*\textup{)},widest = i]
  \item For $Q\in H^1(\Omega, \mcS_0)$ with $Q|_{\partial \Omega} = Q_b$, we have that  
       \begin{equation}\label{eq:lower_bound+}
  \mcE_0 [Q] \geqslant 2 s_+^2 \pi |m|,
    \end{equation}
with equality if, and only if, $Q = s_+(n_0 \otimes n_0 - \third I)$ with $n_0$  conformal and $n_0\cdot e_3$ sign-definite $($i.e., $n_0\cdot e_3$ is either strictly positive or strictly negative$)$.
\medskip
\item  
 The director field $n_0$ is conformal with  $n_0\cdot e_3$ sign-definite if, and only if, its stereographic projection {\textup{(\ref{eq: stereo})}} 
 is given by $w_{0; \bm{a}}$ or by $1/{\overline w}_{0; \bm{a}}$ for some $\bm{a}  \in \Omega^{|m|}$ $($the two alternatives for $n_0$ are related by reflection in $e_3$ $)$.
The planar boundary conditions satisfied by  $n_0$ are given by
\begin{equation}
\label{eq: n_b conformal }
n_{b;\bm{a}} =  \cos \phi_{\bm{a}} \, e_1 + \sin \phi_{\bm{a}} \, e_2, \ \text{ where }\ \phi_{\bm{a}} = \arg w_{0; \bm{a}}.
\end{equation}
The points $\bm{a}$ are precisely the escape points where $n_0 = e_3$ \textup{(}if $n$ has stereographic projection $w_{0; \bm{a}}$\textup{)} or $n_0 = -e_3$ \textup{(}if $n$ has stereographic projection $1/{\overline w}_{0; \bm{a}}$\textup{)}.

  \end{enumerate}
\end{theorem}
Thus, amongst degree-$m$ planar boundary conditions, the leading-order Landau-de Gennes energy achieves its minimum, namely $2\pi |m| s_+^2$, for the $2|m|$-dimensional family  $n_{b;\bm{a}}$, and is independent of the positions $\bm{a}$ of the escape points.   
The proof of Theorem~\ref{thm:locdef} is given in Section~\ref{subsec: conformal - proofs}.

Given $\bm{a}\in \Omega^{|m|}$,  let $Q^*_{\eps;\bm{a}}$ 
denote a minimiser of the Landau-de Gennes energy  subject to boundary conditions \eqref{eq: orientable bcs} with boundary director $n_{b; \bm a}$ given by \eqref{eq: n_b conformal }.    From Proposition~\ref{prop:zero} and Theorem~\ref{thm:locdef},  we have that $Q^*_{\eps;\bm{a}} \rightarrow  Q[n_{0;\bm{a}}]$ as $\eps \rightarrow 0$. 
From Theorems~\ref{thm:gexp2} and \ref{thm:locdef}, we have that
\begin{equation}
 \label{eq: E_0 and W_LdG}
\frac{1}{s_+^2} \mcE_\eps[Q_{\eps; \bm{a}}] = 2\pi |m| + \eps^2 W_{LdG}(\bm{a}) + o(\eps^2), \ \text{ where }W_{LdG}(\bm{a})  
= -\frac{3}{\nu} \|\nabla n_{0; \bm{a}}\|_{L^4}^4. 
\end{equation}
The above energy expression provides a tool to distinguish between various conformal configurations using locations of escape points.  Let us  examine  how the first-order energy, $W_{LdG}(\bm{a})$, depends on $\bm{a}$.  
Since the  $L^2$-norm of $\nabla n_{0; \bm{a}}$ is fixed (its square is equal to  $2\pi|m|$), it follows that $W_{LdG}(\bm{a})$ decreases as $\nabla n_{0; \bm{a}}$ becomes more concentrated.  Concentration occurs as the escape points move towards the boundary, since $n_{0; \bm{a}} = \pm e_3$ at escape points while $n_{0; \bm{a}}\cdot e_3 = 0$ at the boundary.

One can show that as the  distance $\delta := \min_j \text{dist}\,(a_{(j)},\partial \Omega)$ goes to zero, $W_{LdG}(\bm{a})$ diverges as $\delta^{-2}$.  This is compatible with Theorem~\ref{thm:gexp2}, which concerns the behaviour of the energy as $\eps\rightarrow 0$ for fixed boundary conditions.  To analyse the energy for $ \eps, \delta \rightarrow 0$ simultaneously, one would need to go to higher order in the $\Gamma$-expansion and include a boundary-layer analysis. 

In the case of the two-disk $\Omega = D^2$, $g_a$ and $h_a$ are given by
\begin{equation}
\label{eq: g_a and h_a for disk }
g_a(x) + ih_a(x) = \frac{1}{1 -ax}.
\end{equation}
In this case, if ${\bm{a} = 0}$, i.e., if the escape points coincide at the origin, then the conformal boundary condition $n_{b}$  is {\it $m$-radial} \cite{Ignat2015binstability, kitavtsev2016liquid}, and 
\begin{equation}
\label{eq: n_b radial }
n_{b} =  \cos (m\varphi + \alpha) \, e_1 + \sin (m\varphi + \alpha) \, e_2,
\end{equation}
where $\varphi$ is the polar angle coordinate on $\RR^2$ and $\alpha$ is a constant.

\begin{remark}\label{rem: no biaxiality for conformal}
Let $Q^*_\eps$ denote a Landau-de Gennes minimiser with conformal leading-order Oseen-Frank director $n_0$. It follows from
\eqref{eq:  P_0^N expand}-\eqref{eq: cj's3 } and Definition~\ref{def: conformality}
that 
$Q^*_\eps - Q[n^*_\eps]$ is proportional to $Q[n^*_\eps]$ to leading order; that is,  the induced biaxiality in $Q^*_\eps$ does not appear at $O(\eps^2)$ but at higher order.
\end{remark}

\begin{figure}[tb]
    \centering
    \captionsetup{width=1\linewidth}
    \begin{subfigure}[b]{0.48\textwidth}
        \includegraphics[width=\textwidth]{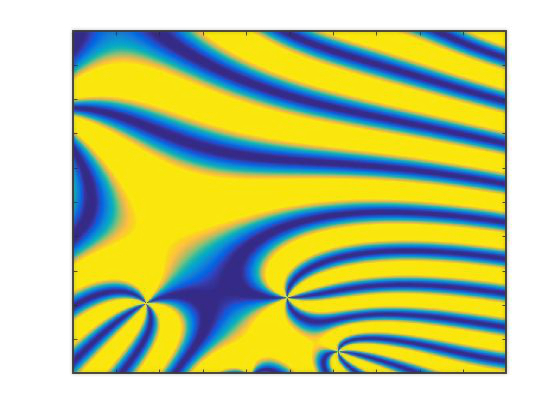}
        \caption{Conformal director -- global minimiser}
       \label{fig:S1}
    \end{subfigure}
    ~ 
    \begin{subfigure}[b]{0.48\textwidth}
        \includegraphics[width=\textwidth]{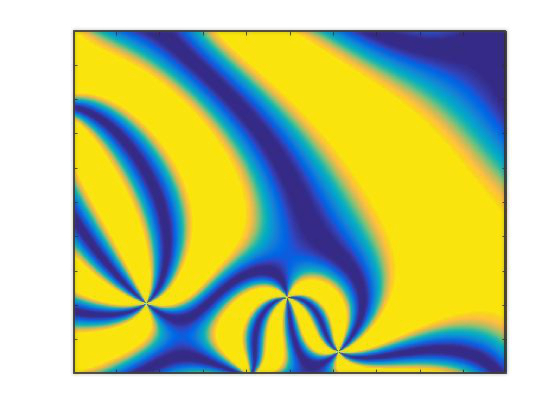}
        \caption{Conformal director -- local minimiser}
        \label{fig:S2}
    \end{subfigure}
     ~ 
    \caption{  \small Schlieren textures in  conformal director fields. 
    The colour scale corresponds to the quantity $[n_1 n_2 /(n_1^2 + n_2^2)]^2$, where $n_j := n\cdot e_j$;  this quantity is proportional to the intensity of light passing through a nematic film with director $n(x,y)$ placed between  polarisers with polarisation axes $e_1$ and $e_2$.  Random conformal director fields were constructed from (a) Eq.~\eqref{eq: w} and (b) Eq.~\eqref{eq: mixed w } by setting $g = h = 0$, corresponding to an infinite planar domain, with  escape points chosen at random in a large region of the plane, one portion of which is shown in the figures. In (a), the director is equal to $+e_3$ at all escape points, 
 while in (b), the director is randomly taken to be $+e_3$ or $-e_3$  at escape points.
    Note that  in (b), contours (lines of constant hue) can join pairs of escape points, but not in (a).  This can be understood in terms of the analytic representations \eqref{eq: w} and \eqref{eq: mixed w }.  Contour lines correspond to lines on which $\arg w$ is fixed, which are also lines of steepest descent of $|w|$.   Escape points with $n = e_3$ or $n = -e_3$ correspond respectively to zeros or poles of $w$.  A zero and a pole of $w$ can be joined by a line of steepest descent, but two zeros  of $w$ cannot, nor can two poles.} \label{fig:Schlieren}
\end{figure}

Let us indicate a generalisation of Theorem~\ref{thm:locdef}.  The space of director fields $n \in H^1(\Omega, \Sphere^2)$ satisfying planar boundary conditions can be partitioned into homotopy classes $(r,s)$ labeled by a pair of integers.  For $n$ differentiable, $r$ and $s$ correspond respectively to a signed count of the preimages of regular values of $n$ in the northern and southern hemispheres of $\Sphere^2$, with the sign given by the sign of the determinant of the Jacobian at the preimage.   The director field with stereographic projection $w_{\bm{0;a}}$ belongs to the class $(m,0)$ for $m >0$ and to $(0,-m)$ for $m<0$.  Its reflection in $e_3$, which has  stereographic projection $1/{\overline w}_{\bm{0;a}}$, belongs to the class $(-m,0)$ for $m >0$ and to $(0,m)$ for $m<0$.  For a general class $(r,s)$, the degree of the planar boundary conditions is given by $m = r-s$.
%
%
It is straightforward to show (for $C^1$-boundary conditions) that for $n$ in the class $(r,s)$, the one-constant Oseen-Frank energy $\mcE_{OF}(n)$ is bounded below by $2\pi s_+^2( |r| + |s|)$ - this generalises the first assertion in Theorem~ \ref{thm:locdef}.

The second assertion may be generalised as follows:
For $r$ and $s$ non-negative, conformal directors in the homotopy class $(r,s)$ that saturate the lower bound are given by
\begin{equation}
\label{eq: mixed w }
 w_{0;\bm{b, c}}(x) =  e^{i\alpha}\prod_{j=1}^{|r|} \frac{x-b_{(j)}}{ \exp\bigl( g_{b_{(j)}} + i  h_{b_{(j)}}\bigr)} 
 \prod_{k=1}^{|s|} \frac {\exp\bigl( g_{c_{(k)}} + i  h_{c_{(k)}}\bigr)}{x-c_{(k)}} \, ,
 \end{equation} 
where $\bm{b}$ and $\bm{c}$ are respectively $|r|$- and $|s|$-tuples of points in $\Omega$.  The $b_{(j)}$'s are the points where $n_0 = e_3$, and the 
$c_{(k)}$'s are the points where $n_0 = -e_3$.  For $r$ (resp.~$s$) negative, the first (resp.~second) product in \eqref{eq: mixed w } is replaced by its complex conjugate.  These are local minimizers of the Dirichlet energy with respect to their boundary conditions (they are global minimisers for $r=0$ or $s=0$).
%
Director fields corresponding to  \eqref{eq: w} and \eqref{eq: mixed w } are shown in Figure~\ref{fig:Schlieren}.
\subsection{Applications to conformal director fields: proofs}\label{subsec: conformal - proofs}
\begin{proof}[Proof of Proposition ~\ref{prop: conformal implies harmonic}] Note that $n$ conformal implies that $\partial_2 n = \sigma n \times \partial_1 n$ and $\partial_1 n = - \sigma n \times \partial_2 n$. Therefore, for $\phi \in {\mycal D}(\Omega, \RR^3)$, we have that
\begin{align} \label{eq: 411}
\int_\Omega \Delta n \cdot \phi  &
%
= \sigma \int_\Omega (\partial_1 \phi \times  n) \cdot \partial_2 n - (\partial_2 \phi \times  n) \cdot \partial_1 n \nonumber \\
&= \sigma \int_\Omega \left( \partial_1( \phi \times  n)  - \phi\times \partial_1 n\right) \cdot \partial_2 n - \left( \partial_2( \phi \times  n)  - \phi\times \partial_2 n\right) \cdot \partial_1 n. 
\end{align}
We note that $m := \phi\times n \in H^1_0(\Omega, \RR^3)$, so that
\be \label{eq: 412}\int_\Omega \partial_1 m  \cdot \partial_2 n -   \partial_2 m  \cdot \partial_1 n = 0. \ee
From \eqref{eq: 411} and  \eqref{eq: 412},
\be \int_\Omega \Delta n \cdot \phi = 2 \sigma \int_\Omega (\phi\times \partial_2 n) \cdot \partial_1 n = -2\sigma\int_\Omega (\partial_1 n \times \partial_2 n)\cdot \phi.\ee
The fact that $n$  is conformal implies that $\partial_1 n \times \partial_2 n = \half \sigma |\nabla n|^2 n$, 
from which it follows that $n$ is a weakly harmonic map, i.e., $\Delta n = -|\nabla n|^2 n$.  From the regularity result of H\'elein \cite{Helein2d}, it follows that $n$ is real analytic.
\end{proof}

\label{subsec: proof of harmonic Q = conformal n}
\begin{proof}[Proof of Proposition ~\ref{prop: harmonic uniaxial Q-tensor}]
First,  suppose that $n\in H^1(\Omega, \Sphere^2)$ is conformal.  From Proposition~\ref{prop: conformal implies harmonic}, we have  that $n$ is a real analytic $\Sphere^2$-valued harmonic map.  
Let \be \label{eq: Prop 2.2 Q def} Q = \sqrt{3/2} \left(n \otimes n - \third I \right) \in C^\infty(\Omega,\Sphere^4).\ee   
Using the harmonic map equation for $n$, we have that 
\begin{align}
\Delta Q & = \sqrt{3/2} \left(\Delta n \otimes n + 2 \grad n \otimes \grad n + n \otimes \Delta n\right) \\
& = 
-\sqrt{6} \left (|\nabla n|^2 n \otimes n  - \grad n \otimes \grad n \right). \label{eq: pf P2.1 Delta Q}
\end{align}
Also, $n$ conformal implies that  $ \partial_1 n\cdot \partial_2 n = 0$ and $|\partial_1 n| = |\partial_2 n| $.  Therefore,  if $\lambda :=  |\nabla n|/\sqrt2 \neq 0$, then the three unit-vectors 
$\lambda^{-1} \partial_1 n$,   $ \lambda^{-1} \partial_2 n$ and $n$ constitute an orthonormal frame.  It follows that
\begin{equation}
\label{eq: P2.1 n_, projection }
\grad n \otimes \grad n = \half |\nabla n|^2 (I - n\otimes n).
\end{equation}
Substituting  \eqref{eq: P2.1 n_, projection } into \eqref{eq: pf P2.1 Delta Q}, we get that
\begin{equation}
\label{eq: pf P2.1 Delta Q 2}
\Delta Q =  - 3\sqrt{3/2}  |\nabla n|^2  \left(n\otimes n - \third I\right) = -3|\nabla n|^2 Q = -|\nabla Q|^2 Q,  
\end{equation}
as $|\nabla Q|^2 = 3 |\nabla n|^2$.   Thus, $Q$ is an $\Sphere^4$-valued harmonic map.

Next, let $Q \in H^1(\Omega, \Sphere^4)$ be given by \eqref{eq: Prop 2.2 Q def} with $n \in H^1(\Omega, \Sphere^2)$, and suppose $Q$ is
an $\Sphere^4$-valued harmonic map.  Then $Q$ is real analytic \cite{Helein2d}, which implies that $n$ is real analytic.   From the harmonic map equation for $Q$, we get that 
\begin{equation}
\label{eq: P.21 Q->n 1 }
\Delta n \otimes n + 2 \grad n \otimes \grad n + n \otimes \Delta n = -3|\nabla n|^2 \left( n\otimes n - \third I\right).
\end{equation}
Applying both sides of the preceding equation to $n$ and using the identities  $\partial_i n\cdot n = 0$,  $i = 1,2$ and $ \Delta n \cdot n = -|\nabla n|^2$, which follow from $|n| = 1$, we get that $n$ is a harmonic map, i.e.,
$\Delta n = -|\nabla n|^2 n$.  Substitution of this relation into \eqref{eq: P.21 Q->n 1 } yields
\begin{equation}
\label{eq: P.21 Q->n 2 }
 2 \grad n \otimes \grad n = |\nabla n|^2 (I - n\otimes n).
\end{equation}
Applying both sides of the preceding equation to $\partial_1 n$ and $\partial_2 n$ yields the pair of vector equations
\begin{equation}
\label{eq: P.21 end }
\alpha \partial_1 n + \beta \partial_2 n = \beta \partial_1 n + \gamma \partial_2 n = 0, \end{equation}
where $\alpha = |\partial_1 n|^2 - \half |\nabla n|^2$, $\beta = \partial_1 n\cdot \partial_2 n$, and $\gamma = |\partial_2 n|^2 - \half |\nabla n|^2$.
The solvability conditions are 
$\alpha = \beta = \gamma = 0$, which are equivalent to the condition \eqref{eq: conformal director} for $n$ to be conformal.
\end{proof}

\begin{proof}[Proof of Theorem~\ref{thm:locdef}]  (\emph{i}) Without loss of generality we may assume that   $Q \in H^1(\Omega, \mcS_*)$, since otherwise $\mcE_0(Q) =+ \infty$.  Since $\Omega$ is simply connected, it follows that $Q = s_+(n\otimes n - \third I)$ for some $n \in H^1(\Omega, \Sphere^2)$.  Since we are seeking to establish a lower bound for the energy, we can assume without loss of generality that $Q$ is global minimiser of $\mcE_0$.  From Remark~\ref{rem: unique minimiser Q_0},  it follows that $n$ is a minimising $\Sphere^2$-valued harmonic map, and without loss of generality we may assume that $n\cdot e_3 > 0$. The classical regularity result of H\'elein \cite{Helein2d} on two-dimensional harmonic maps implies that $n$ is smooth up to the boundary.
The following bound is standard (see, for instance, \cite{BrCo83}):  
\begin{align}
 \mcE_0 [Q] & = s_+^2 \int_{\Omega} |\partial_1 n|^2 +  |\partial_2 n|^2 \geqslant 2  s_+^2 \int_{\Omega}  |\partial_1 n| \,  |\partial_2 n|  \nonumber \\
 & \qquad \qquad \geqslant 2  s_+^2 \left| \int_{\Omega} n \cdot  \left(\partial_1 n   \times  \partial_2 n\right) \right| = 2  s_+^2 \left| {\mcA}[n(\Omega)] \right|,
 \label{eq: energy inequalities }
 \end{align}
where $\mcA(n(\Omega))$ denotes  the oriented area  $n(\Omega) \subset \Sphere^2$.  For completeness, we provide an argument.  Let us introduce spherical polar coordinates for $n$,
   \begin{equation}
 \label{eq: spherical coords}
n = \sin\theta \cos \varphi\, e_1 + \sin\theta \sin \varphi\, e_2 + \cos\theta e_3,
\end{equation}
and similarly for the  $C^1$-boundary conditions, $n_b = n|_{\partial \Omega} =  \cos \varphi_b e_1 + \sin \varphi_b e_2$.
We may express the oriented area in terms of spherical polar coordinates as
 \begin{equation}\label{eq: oriented area 1}
 {\mcA}[n(\Omega)] =  \int_{\Omega} n \cdot  \left(\partial_1 n   \times  \partial_2 n\right) =  \int_{\Omega} \sin \theta \, \left(\partial_1\theta \partial_2\varphi - \partial_2\theta \partial_1\varphi\right).
 \end{equation}
Let  
\begin{equation}
\label{eq: F }
F = (1-\cos \theta)(\partial_2\varphi e_1 -\partial_1\varphi e_2).
\end{equation}
Since $n$ is smooth, $F$ is smooth; this is in spite of the fact that $\nabla \varphi$  may have singularities where $\theta = 0$ or $\theta =\pi$, since $F$ vanishes if $\theta = 0$ while $\theta = \pi$ is excluded by $n\cdot e_3 > 0$.  Noting that 
$\sin \theta \, (\partial_1\theta \partial_2\varphi - \partial_2\theta \partial_1\varphi) =\nabla\cdot F$, we apply the divergence theorem in \eqref{eq: oriented area 1} to obtain
\begin{equation}
{\mcA}[n(\Omega)] =  \int_{\partial\Omega} F\cdot \nu =\int_{\partial\Omega} \varphi_b' = 2\pi m ,
 \end{equation}
where $\nu$ denotes the unit normal on $\partial \Omega$, $\varphi_b'$ denotes the tangential derivative of $\varphi_b$, and
$m$ is the degree of $\exp(i\varphi)$, regarded as an $\Sphere^1$-valued map on $\partial \Omega$.    This establishes the lower bound \eqref{eq:lower_bound+}. 

The first inequality in \eqref{eq: energy inequalities } is saturated  if and only if  $|\partial_1 n| =  |\partial_2 n|$, and the second inequality is saturated  if and only if $\partial_1 n$ and $\partial_2 n$ are orthogonal.  As $n$ is orthogonal to  both $\partial_1 n$ and $\partial_2 n$, 
these two conditions are equivalent to the condition
 \begin{equation}
\label{eq: equivalent to (anti)con }
\partial_2 n = \sigma n  \times \partial_1 n, \ \ \sigma = \pm 1.
\end{equation}
The last inequality in \eqref{eq: energy inequalities } is saturated if and only if $\sigma$ is  constant, i.e., with regard to Definition~\ref{def: conformality} if and only $n$ is conformal.

\vskip 0.2cm

\noindent{(\emph{ii})} 
We are given that $n\in C^\infty(\Omega, \Sphere^2)$ is a conformal minimising $\Sphere^2$-valued harmonic map with degree-$m$ planar  $C^1$-boundary conditions $n_b = \cos\varphi_b e_1 +  \cos\varphi_b e_2$.  We will obtain an explicit formula for $n$ in terms of its escape points, i.e., points where $n$ is parallel to $e_3$, and thereby determine the special form that  $n_b$ must assume.
For definiteness, we take $m$ positive and (cf.~Remark \ref{rem: unique minimiser Q_0}) $n\cdot e_3 > 0$, which together imply that $\sigma = 1$ in 
\eqref{eq: equivalent to (anti)con }.  The adjustments required for the alternative cases are explained at the end.


For $a\in \Omega$, we denote by $g_a$ the solution of the Laplace equation \eqref{eq: g },
and we let $h_a$ denote a harmonic conjugate of $g_a$.  Then $g_a + i h_a$ is holomorphic on $\Omega$.
Let $w$ denote the stereographic projection of $n$, as in \eqref{eq: stereo}.  It is straightforward to verify that the conformal condition \eqref{eq: equivalent to (anti)con }  is equivalent to the Cauchy-Riemann equations
\begin{equation}
\label{eq: Caucy-Riemann }
\Re \partial_1 w  =  \Im \partial_2 w, \quad \Re \partial_2 w = -\Im \partial_1 w.
\end{equation}
Also, $n\cdot e_3>0$ implies that $w$ is bounded.   Therefore, $w$ is complex holomorphic on $\Omega$.    We have that  $\Im \int_{\partial \Omega} d \log w = 
\int_{\partial \Omega} \varphi_b' =  2\pi m$.  It follows that $w$ has precisely $m$ zeros in $\Omega$, counted with multiplicity.  Let $\bm{a} = (a_{(1)},\ldots, a_{(m)}) \in \Omega^m$ denote these zeros, and let
\begin{equation}
\label{eq: f on Omega }
f = w \prod_{j=1}^m \frac{\exp\bigl(
g_{a_{(j)}} + i h_{a_{(j)}}
\bigr)}{
x - a_{(j)}}.
\end{equation}
Then $f$ is holomorphic and nonvanishing on $\Omega$.  It follows that $\log f$ is holomorphic on $\Omega$, so that $\Re \log f$ is harmonic, i.e., $\Delta (\Re \log f) = 0$.  Also, since $|w| = 1$ on $\partial \Omega$, it follows that  $\Re \log f$ vanishes on $\partial \Omega$.  But then $\Re \log f$ must vanish identically, which implies that $\Im \log f$ is constant, i.e.,
$f = \exp(i\alpha)$ for some $\alpha\in \RR$.  Therefore,
\begin{equation}
\label{eq: w formula }
w = e^{i\alpha} \prod_{j=1}^m 
\frac
{x - a_{(j)}}
{\exp\bigl( g_{a_{(j)}} + i h_{a_{(j)}}\bigr)},
\end{equation}
which is equivalent to \eqref{eq: w} for $m > 0$ and $n\cdot e_3  > 0$.
The boundary condition \eqref{eq: n_b conformal } is obtained by setting $x \in \partial \Omega$ and stereographic projection.

The transformation $m \mapsto -m$ while leaving $n\cdot e_3$ unchanged is achieved by $w \mapsto \overline{w}$; we note that $\overline{w}$ is antiholomorphic.  The transformation $(n_1,n_2, n_3) \mapsto (n_1, n_2, -n_3)$ while leaving $m$ unchanged is achieved by $w \mapsto 1/\overline{w}$; we note that $1/\overline{w}$ is antimeromorphic with poles but no zeros.  Finally, simultaneously changing the signs of $m$ and $n\cdot e_3$ is achieved by  $w \mapsto 1/w$.
\end{proof}


\begin{remark}
The lower bound \eqref{eq:lower_bound+} can be established for general $H^1$ maps \textup{(}thus bypassing the regularity result of Helein {\textup{\cite{Helein2d}}}\textup{)}  by performing the arguments in the proof for smooth maps and using the density of smooth maps into $H^1$ maps for $2d$ domains \textup{(}see Schoen and Uhlenbeck {\textup{\cite{ScUh83})}}.
\end{remark}

\subsection{The case $b^2=0$}\label{subsec: b = 0}
 For $b^2 = 0$, we have from Eq.~\eqref{eq: form of Q_0 for b = 0} that the Landau-de Gennes minimiser is given to leading order by $\sqrt{2/3}\, s_+ \sum_{j=1}^3 (c_0\cdot e_j) F_j$, where $c_0 \in H^1(\Omega, \Sphere^2)$ is (weakly) harmonic.
In analogy with the $b^2 \neq 0$ case, we can obtain explicit results for a special family of planar boundary conditions for which $c_0$ is conformal. In this case, the escape points, which parameterise the family, are points where $Q_0$ is oblate uniaxial (rather than prolate uniaxial) with director $e_3$.

 \begin{theorem}\label{thm:locdef b=0} Let $\Omega\subset\RR^2$ be a bounded, simply-connected domain with $C^1$ boundary, and let $Q_b \in C^1 (\partial\Omega, \mcB_Q)$ be a degree-$k$ uniaxial planar $Q$-tensor field on the boundary $\partial \Omega$. 
 \begin{enumerate}[leftmargin=0.8cm,label=\textup{(}\roman*\textup{)},widest = i]
  \item For $Q\in H^1(\Omega, \mcS_0)$ with $Q|_{\partial \Omega} = Q_b$,  we have that 
       \begin{equation}\label{eq:lower_bound+ b=0}
  \mcE_0 [Q] \geqslant \frac{4}{9} s_+^2 \pi |k|,
    \end{equation}
with equality if and only if 
\be
Q = \sqrt{2/3} s_+ \sum_{j=1}^3 (c_0\cdot e_j) F_j
\ee
 and $c_0$ is conformal with $c_0\cdot e_3 < 0$.
\smallskip
\item The field $c_0$ is conformal with $c_0\cdot e_3 < 0$ if and only its stereographic projection \eqref{eq: stereo} is given by 
\begin{equation}
\label{eq: w_pm b=0}
w_{0; \bm{a}}(x) = \sqrt{3} 
 \prod_{j=1}^{|k|}   \frac{ \exp\bigl( g_{a_{(j)}} + i  h_{a_{(j)}}\bigr)}{x-a_{(j)}}, \ \  \bm{a} = (a_{(1)}, \ldots, a_{(|k|)}) \in \Omega^{|k|},
\end{equation} 
for $k > 0$, and by $\overline{w}_{0; \bm{a}}$ for $k < 0$.
The corresponding boundary conditions are given by $Q_{b; \bm{a}} = s_+(n_{b;\bm{a}}\otimes n_{b; \bm{a}} - \third I)$, where 
\begin{equation}
\label{eq: n_b conformal b=0}
n_{b;\bm{a}} =  \cos \phi_{\bm{a}} \, e_1 + \sin \phi_{\bm{a}} \, e_2, \quad \phi_{\bm{a}} =  \half \sgn{k} \sum_{j=1}^{|k|} \arg (x-a_{(j)}) - h_{a_{(j)}}.
\end{equation}
  \end{enumerate}
\end{theorem}

The proof is essentially the same as for Theorem~\ref{thm:locdef}, and hence is omitted.  We note the two  different ways in which an $\Sphere^2$-valued harmonic map is associated with a $Q$-tensor field, namely quadratically via \eqref{eq: form of Q_0 for uniaxial}  for uniaxial $Q$-tensors  when $b^2\neq 0$, and linearly via  \eqref{eq: form of Q_0 for b = 0} when $b=0$. The latter allows for the representation of non-orientable boundary conditions. 
\section{Acknowledgements}
GDF acknowledges support from the Austrian Science Fund (FWF)
through the special research program {\emph{Taming complexity in partial
differential systems}} (Grant SFB F65) and of the Vienna Science and
Technology Fund (WWTF) through the research project {\emph{Thermally
controlled magnetization dynamics}} (Grant MA14-44).  JR and VS acknowledge support from 
EPSRC grant EP/K02390X/1 and Leverhulme grant RPG-2014-226, and JR acknowledges support from a Lady Davis Visiting Professorship at the Hebrew University.

The work AZ is supported by the Basque Government through the BERC 2018-2021
program, by Spanish Ministry of Economy and Competitiveness MINECO through BCAM
Severo Ochoa excellence accreditation SEV-2017-0718 and through the project MTM2017-82184-R, acronym ``DESFLU'', funded by (AEI/FEDER, UE).
\par The authors would like to thank the Isaac Newton Institute for Mathematical Sciences for support and hospitality during the programme { ``\emph{The design of new materials programme}"} when work on this paper was undertaken.

This work was supported by: EPSRC grant numbers EP/K032208/1 and EP/R014604/1.

\bibliographystyle{acm}
\bibliography{DFRSZ-LdG_OF}
\end{document}